\documentclass[oneside,french,english]{amsart}
\usepackage[T1]{fontenc}
\usepackage[latin9]{inputenc}
\usepackage{geometry}
\usepackage{color}
\usepackage{babel}
\makeatletter
\addto\extrasfrench{%
   \providecommand{\fg}{\ifdim\lastskip>\z@\unskip\fi~\frqq}%
}

\makeatother
\usepackage{units}
\usepackage{amsthm}
\usepackage{amstext}
\usepackage{amssymb}
\usepackage{xargs}[2008/03/08]
\usepackage[unicode=true,pdfusetitle,
 bookmarks=true,bookmarksnumbered=false,bookmarksopen=false,
 breaklinks=true,pdfborder={0 0 1},backref=false,colorlinks=true]
 {hyperref}
\usepackage{breakurl}

\makeatletter

\newcommand{\noun}[1]{\textsc{#1}}

\numberwithin{equation}{section}
\numberwithin{figure}{section}
 \newcommand\thmsname{\protect\theoremname}
 \newcommand\nm@thmtype{theorem}
 \theoremstyle{plain}
 
 \newenvironment{namedthm}[1][Undefined Theorem Name]{
   \ifx{#1}{Undefined Theorem Name}\renewcommand\nm@thmtype{theorem*}
   \else\renewcommand\thmsname{#1}\renewcommand\nm@thmtype{namedtheorem}
   \fi
   \begin{\nm@thmtype}}
   {\end{\nm@thmtype}}
  \theoremstyle{remark}
  \newtheorem*{rem*}{\protect\remarkname}
\newenvironment{lyxlist}[1]
{\begin{list}{}
{\settowidth{\labelwidth}{#1}
 \setlength{\leftmargin}{\labelwidth}
 \addtolength{\leftmargin}{\labelsep}
 }}
{\end{list}}
\theoremstyle{plain}
\newtheorem{thm}{\protect\theoremname}[section]
  \theoremstyle{plain}
  \newtheorem{cor}[thm]{\protect\corollaryname}
  \theoremstyle{plain}
  \newtheorem{lem}[thm]{\protect\lemmaname}
  \theoremstyle{definition}
  \newtheorem{defn}[thm]{\protect\definitionname}
  \theoremstyle{plain}
  \newtheorem{prop}[thm]{\protect\propositionname}
  \theoremstyle{remark}
  \newtheorem{rem}[thm]{\protect\remarkname}

\usepackage{kpfonts}
\usepackage{esint}

\AtBeginDocument{
  
}

\makeatother

  \addto\captionsenglish{\renewcommand{\corollaryname}{Corollary}}
  \addto\captionsenglish{\renewcommand{\definitionname}{Definition}}
  \addto\captionsenglish{\renewcommand{\lemmaname}{Lemma}}
  \addto\captionsenglish{\renewcommand{\propositionname}{Proposition}}
  \addto\captionsenglish{\renewcommand{\remarkname}{Remark}}
  \addto\captionsenglish{\renewcommand{\theoremname}{Theorem}}
  \addto\captionsfrench{\renewcommand{\corollaryname}{Corollaire}}
  \addto\captionsfrench{\renewcommand{\definitionname}{Définition}}
  \addto\captionsfrench{\renewcommand{\lemmaname}{Lemme}}
  \addto\captionsfrench{\renewcommand{\propositionname}{Proposition}}
  \addto\captionsfrench{\renewcommand{\remarkname}{Remarque}}
  \addto\captionsfrench{\renewcommand{\theoremname}{Théorème}}
  \providecommand{\corollaryname}{Corollary}
  \providecommand{\definitionname}{Definition}
  \providecommand{\lemmaname}{Lemma}
  \providecommand{\propositionname}{Proposition}
  \providecommand{\remarkname}{Remark}
  \providecommand{\theoremname}{Theorem}
\providecommand{\theoremname}{Theorem}

\begin{document}
\global\long\def\ww#1{\mathbb{#1}}

\global\long\def\germ#1{\ww C\left\{  #1\right\}  }

\global\long\def\mero#1{\ww C\left\{  \left(#1\right)\right\}  }

\global\long\def\fmero#1{\ww C\left(\left(#1\right)\right)}

\global\long\def\frml#1{\ww C\left[\left[#1\right]\right]}

\global\long\def\pol#1#2{\ww C_{#1}\left[#2\right]}

\global\long\def\ratio#1{\ww C\left(#1\right)}

\global\long\def\pp#1{\frac{\partial}{\partial#1}}

\global\long\def\flw#1#2#3{\Phi_{#1}^{#2}#3}

\global\long\def\lie#1{\mathcal{L}\left(#1\right)}

\global\long\def\fol#1{\mathcal{F}_{#1}}

\global\long\def\trs{\pitchfork}

\global\long\def\zsk{\mathbb{Z}/k}

\global\long\def\ddt{\cdot}

\global\long\def\rrel{\mathcal{R}}

\global\long\def\surj{\twoheadrightarrow}

\global\long\def\inj{\hookrightarrow}

\global\long\def\void{\emptyset}

\global\long\def\tx#1{\mathrm{#1}}

\global\long\def\tt#1{\mathtt{#1}}

\global\long\def\nf#1#2{\nicefrac{#1}{#2}}

\global\long\def\ii{\tx i}

\global\long\def\dd#1{\tx d#1}

\global\long\def\gal#1{\tt{Gal}\left(#1\right)}

\global\long\def\dlie#1{\frak{Gal}\left(#1\right)}

\newcommandx\dtrans[1][usedefault, addprefix=\global, 1=Z]{\frak{Trans}\left(#1\right)}

\newcommandx\dtang[1][usedefault, addprefix=\global, 1=Z]{\frak{Tang}\left(#1\right)}

\global\long\def\aut#1{\tt{Aut}\left(#1\right)}

\newcommandx\sol[2][usedefault, addprefix=\global, 1=Z, 2=\delta]{\tt{Sol}_{#2}\left(#1\right)}

\newcommandx\tang[1][usedefault, addprefix=\global, 1=Z]{\tt{Tang}\left(#1\right)}

\newcommandx\trans[1][usedefault, addprefix=\global, 1=Z]{\tt{Trans}\left(#1\right)}

\title{Meromorphic infinitesimal affine actions of the plane{*}}

\selectlanguage{french}%

\address{Laboratoire I.R.M.A., 7 rue R. Descartes, Université de Strasbourg,
67084 Strasbourg cedex, France}

\selectlanguage{english}%

\email{\href{mailto:teyssier@math.unistra.fr}{teyssier@math.unistra.fr}}

\urladdr{\href{http  ://math.u-strasbg.fr/~teyssier/}{http  ://math.u-strasbg.fr/$\sim$teyssier/}}

\author{Loïc TEYSSIER}

\date{November 2013}
\begin{abstract}
We study complex Lie algebras spanned by pairs $\left(Z,Y\right)$
of germs of a meromorphic vector field of the complex plane satisfying
$\left[Z,Y\right]=\delta Y$ for some $\delta\in\ww C$. This topic
relates to Liouville\textendash{}integrability of the differential
equation induced by the foliation underlying $Z$. We give a direct
geometric proof of a result by \noun{M.~Berthier} and \noun{F.~Touzet}
characterizing germs of a foliation admitting a first\textendash{}integral
in a Liouvillian extension of the standard differential field. In
so doing we study transverse and tangential rigidity properties when
$Z$ is holomorphic and its linear part is not nilpotent. A second
part of the article is devoted to computing the Galois\textendash{}Malgrange
groupoid of meromorphic $Z$.
\end{abstract}

\thanks{{*}PrePrint}

\maketitle

\section{Introduction}

Most reduced singularities of holomorphic foliations $\fol{}$ in
the complex plane are locally analytically conjugate to a simple polynomial
model, either its linear part or a Dulac\textendash{}Poincaré normal
form~\cite{Dulac}, all of whom admit Liouvillian first\textendash{}integrals.
The other cases consist in quasi\textendash{}resonant saddles (small
divisors problem) and resonant singularities (encompassing resonant
saddles and saddle\textendash{}nodes). \noun{M.~Berthier} and \noun{F.~Touzet}
have studied in~\cite{BerTouze} these difficult cases, characterizing
those admitting a Liouvillian first\textendash{}integral $H$. By
carefully studying the transverse dynamics they worked out that $\fol{}$
needs to be of a certain type. For quasi\textendash{}resonant singularities
$H$ is of moderate growth along the separatrices of the foliation
(Nilsson class); in that case $\fol{}$ is analytically linearizable
(see~\cite{Cervo} as well). For resonant singularities at least
one separatrix of $\fol{}$ is a line of essential singularities for
$H$; here $\fol{}$ is conjugate to a rather simple model and the
Martinet\textendash{}Ramis modulus~\cite{MaRa-Res,MaRa-SN} is an
affine map (in a suitable ramified chart). In any cases the local
form of $\fol{}$ is induced by a first\textendash{}order Bernoulli
differential equation. This article mainly provides a geometrical
proof of that fact, simplifying the arguments through the use of a
unified framework which is more explicit and does not need the study
of functional moduli or transverse dynamics. This formalism focuses
on meromorphic vector fields $Z$ tangent to $\fol{}$ and allows
to provide normal forms for holomorphic ones. We are also interested
in computing the Galois\textendash{}Malgrange groupoid~\cite{Casa,Malg}
associated to $Z$.

\bigskip{}

Our main object of study are Lie algebras $\lie{Z,Y}$ over $\ww C$
generated by two germs at $\left(0,0\right)$ of a non\textendash{}zero,
meromorphic vector field. We will be interested only in finite\textendash{}dimensional
algebras, more precisely those admitting a \textbf{ratio} $\delta\in\ww C$
such that 
\begin{eqnarray*}
\left[Z,Y\right] & = & \delta Y\,.
\end{eqnarray*}
We speak of \textbf{Abelian Lie algebras} when $\delta=0$ and \textbf{affine
Lie algebras} when $\delta\neq0$. In the latter case the ratio may
be chosen equal to $1$ by considering $\frac{1}{\delta}Z$ instead.
The most interesting cases arise when the locus of tangency between
$Z$ and $Y$ is of codimension at least $1$, in which case we shall
write $Z\pitchfork Y$ and speak of \textbf{non\textendash{}degenerate}
affine Lie algebras. These objects correspond to (local infinitesimal)
affine actions of the plane $\ww C^{2}$. Save for explicit mention
to the contrary we shall always assume that $Z\trs Y$. The holomorphic
foliation $\fol Z$, whose leaves are the integral curves of $Z$,
admits a Godbillon\textendash{}Vey sequence~\cite{GodVey} of length
$1$ or 2. It is given by the dual basis $\left(\tau_{Z},\tau_{Y}\right)$
of $\left(Z,Y\right)$: the pair of meromorphic $1$\textendash{}forms
indeed satisfies
\begin{eqnarray*}
\dd{\tau_{Y}} & = & \delta\tau_{Y}\wedge\tau_{Z}\,,\\
\dd{\tau_{Z}} & = & 0\,.
\end{eqnarray*}
 The $1$\textendash{}form $\tau_{Y}$ defines the same foliation
as $Z:=A\pp x+B\pp y$. Because of these properties, the solutions
to the ordinary differential equation 
\begin{eqnarray}
y' & = & \frac{B(x,y)}{A(x,y)}\label{eq:eq_diff}
\end{eqnarray}
are given implicitly by level curves of the first\textendash{}integral
$H$ lying in a Liouvillian extension $\ww L$ over the differential
field $\ww K_{0}:=\mero{x,y}$ of germs at $\left(0,0\right)$ of
a meromorphic function~\cite{Casa,BerTouze}, equipped with the standard
derivations $\left(\pp x,\pp y\right)$. Such an extension is constructed
by taking a finite number of extensions $\ww K_{n+1}:=\ww K_{n}\left\langle f_{n}\right\rangle $
of the following three types : $f_{n}$ is algebraic over $\ww K_{n}$,
a derivative of $f_{n}$ lies in $\ww K_{n}$, a logarithmic derivative
of $f_{n}$ lies in $\ww K_{n}$. The function defined by 
\begin{eqnarray*}
H & := & \int\exp\left(-\delta\int\tau_{Z}\right)\tau_{Y}
\end{eqnarray*}
is a multiform first\textendash{}integral of $Z$ belonging to a Liouvillian
extension of $\ww K_{0}$ (one can check that $\exp\left(-\delta\int\tau_{Z}\right)\tau_{Y}$
is a multiform closed $1$\textendash{}form and that the Lie derivative
$Z\cdot H$ of $H$ along $Z$ vanishes). 

\bigskip{}

A few years ago \noun{B.~Malgrange} gave a candidate Galois theory
for non\textendash{}linear differential equations~\cite{Malg}, by
building an object (a $\mathcal{D}$\textendash{}groupoid) which encodes
information regarding whether or not a differential equation, for
instance of the form~\ref{eq:eq_diff}, can be solved by quadrature.
He proved that in the case of a linear system this groupoid coincides
with the linear group of Picard\textendash{}Vessiot associated to
the system. \noun{G.~Casale} then studied intensively the Galois\textendash{}Malgrange
groupoid for codimension $1$ foliations, and related indeed the existence
of Liouvillian first\textendash{}integrals, the existence of short
Godbillon\textendash{}Vey sequences and the smallness of the groupoid~\cite{Casa}.
We contribute here to this work by computing explicitly the Galois\textendash{}Malgrange
groupoid of the vector field $Z$. This result complements a previous
work of \noun{E.~Paul} about Galois\textendash{}Malgrange groupoids
of quasi\textendash{}homogeneous foliations~\cite{PaulEnv}.

\subsection{Globalizing holomorphic vector fields $Z$}

~

We begin with giving general results on non\textendash{}degenerate
Lie algebras of ratio $\delta$ in Section~\ref{sec:Prelim} when
$Z$ is \emph{holomorphic}. If $Z$ is regular at $\left(0,0\right)$
then $\psi^{*}Z=\pp y$ for some analytic change of local coordinates,
as follows from the theorem of rectification. On the contrary when
$Z\left(0,0\right)=0$ we say that $Z$ is \textbf{non\textendash{}nilpotent}
if its linear part%
\footnote{The linear part of a vector field $\left(ax+by+\cdots\right)\pp x+\left(cx+dy+\cdots\right)\pp y$
is identified with $\left[\begin{array}{cc}
a & b\\
c & d
\end{array}\right]$.%
} at $\left(0,0\right)$ is neither. In that case we label the eigenvalues
$\left\{ \lambda_{1},\lambda_{2}\right\} $ in such a way that $\lambda_{2}\neq0$.
A direct computation (Lemma~\ref{lem:degen}) ensures that $Z$ is
non\textendash{}nilpotent when $\delta\neq0$. 
\begin{namedthm}[Globalization Theorem]
Let $\lie{Z,Y}$ be a non\textendash{}degenerate Lie algebra of ratio
$\delta$ such that $Z$ is non\textendash{}nilpotent. Then $Z$ is
glocal, in the sense that it admits a polynomial form in some local
analytic chart around $\left(0,0\right)$.
\end{namedthm}
If some non\textendash{}nilpotent vector field $Z$ admits a non\textendash{}isolated
singularity at $\left(0,0\right)$ it defines a regular foliation
$\fol Z$ and there exists a local system of coordinates in which
$Z=\lambda_{2}y\pp y$ with $\lambda_{2}\neq0$ (Corollary~\ref{cor:non-isolated}).
We focus next our argument to the case of an isolated singularity.
For most $\lambda:=\nf{\lambda_{1}}{\lambda_{2}}$ it is well known~\cite{Dulac}
that either $Z$ is conjugate to $Z_{0}$ (linearizable) or to a Dulac\textendash{}Poincaré
polynomial form. We will therefore explore only the following non\textendash{}trivial
cases, allowing to cover the remaining cases of non\textendash{}nilpotent
singularities and prove the globalization theorem:
\begin{itemize}
\item quasi\textendash{}resonant saddles ($\lambda\in\ww R_{<0}\backslash\ww Q$),
studied in Section~\ref{sec:irrat}, where we recover the fact that
$Z$ is actually linearizable,
\item resonant singularities ($\lambda\in\ww Q_{\leq0}$, non\textendash{}linearizable),
studied in Section~\ref{sec:res-sad}, where we establish that $Z$
is given in an appropriate local chart by a Bernoulli vector field
with polynomial coefficients.\end{itemize}
\begin{rem*}
~
\begin{enumerate}
\item We give in the course of the article an explicit family of polynomial
normal forms for resonant vector fields meeting the hypothesis of
the theorem, generalizing that of~\cite{Tey-ExSN} obtained for foliations.
\item I do not know whether the globalization theorem still holds when $Z$
is meromorphic. As there exists non\textendash{}glocal foliations~\cite{GenzTey}
this question is non\textendash{}trivial.
\end{enumerate}
\end{rem*}

\subsubsection{Normalization procedure }

~

The techniques of this article can be applied to other non\textendash{}nilpotent
singularities in order to obtain normal forms $\left(Z,Y\right)$
unique up to the action of a finite\textendash{}dimensional space.
This question is non\textendash{}trivial in the presence of a vector
field $Z$ with a non\textendash{}constant meromorphic first\textendash{}integral.
It is possible to reduce $Y$ to a rational vector field using isotropies
of $Z$. These computations are reasonably easy and therefore not
included in the present paper. We nonetheless present the general
procedure of normalization of the pair $\left(Z,Y\right)$ when $Z$
is holomorphic and admits an isolated singularity at $\left(0,0\right)$
with a non\textendash{}nilpotent linear part. We will apply this scheme
for resonant and quasi\textendash{}resonant singularities.
\begin{lyxlist}{00.00.0000}
\item [{\textbf{Preparation}}] We can associate to $Z$ a non\textendash{}degenerate
Lie algebra $\lie{X_{0},Y_{0}}$ of ratio $\delta_{0}$ (the latter
is generically zero and always a holomorphic first\textendash{}integral
of $Y_{0}$) such that in a suitable analytic chart 
\begin{eqnarray*}
Z & = & U\left(X_{0}+RY_{0}\right)
\end{eqnarray*}
for holomorphic $Y_{0}$, $U$ and $R$ where $U(0,0)=1$ and $R$
vanishes sufficiently. It will turn out that $Y_{0}=y\pp y$ or $Y_{0}=\pp y$.
Define
\begin{eqnarray*}
X & := & X_{0}+RY_{0}.
\end{eqnarray*}

\item [{\textbf{Formal~normalization}}] We show that $Z$ is formally
conjugate by $\hat{\psi}$ to a polynomial vector field 
\begin{eqnarray*}
Z_{0} & := & QX_{0}
\end{eqnarray*}
 where $Q(0,0)=1$. Using first a conjugacy formed with the flow along
$Y_{0}$ we conjugate $X_{0}$ to $X$ (we speak of \textbf{transverse
normalization}), then using a conjugacy given by the flow of a suitable
$QX$ we conjugate the latter to $Z$ (we speak of \textbf{tangential
normalization}). 
\item [{\textbf{Rigidification}}] A formal computation ensures that the
\emph{a priori} formal vector field $\hat{Y}:=\hat{\psi}^{*}Y$ is
actually meromorphic. We deduce that $\left(Z,Y\right)$ is analytically
conjugate to some couple $\left(\hat{Z},\hat{Y}\right)$ where 
\begin{eqnarray*}
\hat{Z} & := & Z_{0}+K\hat{Y}
\end{eqnarray*}
with $K$ in the kernel of $\hat{Y}\cdot$. Those are Berthier\textendash{}Touzet
foliations.
\item [{\textbf{Reduction}}] We next use isotropies of $K\hat{Y}$ to reduce
$K$ to a polynomial. Then using isotropies of $\hat{Z}$, we conjugate
$\hat{Y}$ to a global vector field $\tilde{Y}$. 
\end{lyxlist}

\subsubsection{Auxiliary rigidity results}

~

The first result of rigidity addresses the rigidity of tangential
transforms and is of general interest since its scope reaches beyond
the case of transversely affine~/~Abelian vector fields. The following
theorem is a consequence of Corollary~\ref{cor:irrat-tang-rigid}
and Proposition~\ref{pro:res-sad-tang-rigid}.
\begin{thm}
\label{thm:rigid-tang}Assume that $Z$ is a holomorphic vector field
with an isolated, non\textendash{}nilpotent singularity at $\left(0,0\right)$.
There exists a closed meromorphic 1\textendash{}form $\tau$ such
that $\tau(Z)=1$ if, and only if, $Z$ is analytically conjugate
to a vector field $QX$ introduced in the previous paragraph.
\end{thm}
The 1\textendash{}form $\tau$ is called a \textbf{time\textendash{}form}
for $Z$ and we will prove that $QX$ always admits a closed time\textendash{}form:
the 1\textendash{}form $\tau_{0}$ given by the dual basis of $\left(Z_{0},Y_{0}\right)$. 
\begin{cor}
Assume $\lie{Z,Y}$ is a Lie algebra of ratio $\delta$ with $Z$
non\textendash{}nilpotent. If $\delta\neq0$ or $\lie{Z,Y}$ is non\textendash{}degenerate
then $Z$ is analytically conjugate to $QX$.\end{cor}
\begin{proof}
If the Lie algebra $\lie{Z,Y}$ of ratio $\delta\neq0$ is degenerate
then $Y=FZ$ for some meromorphic $F\neq0$ and a direct computation
shows that $F$ satisfies
\begin{eqnarray*}
Z\cdot F & = & \delta F\,.
\end{eqnarray*}
The meromorphic $1$\textendash{}form $\frac{\dd F}{\delta F}$ is
a closed time\textendash{}form for $Z$ so that, according to the
previous theorem, it is analytically conjugate to $QX$. In general
it is not possible to give simpler models. In the case of a non\textendash{}degenerate
affine Lie algebra the dual basis $\left(\tau_{Z},\tau_{Y}\right)$
associated to $\left(Z,Y\right)$ provides a closed time\textendash{}form
$\tau_{Z}$ for $Z$, which is therefore analytically conjugate to
$QX$.
\end{proof}
\bigskip{}
We then use transverse changes of coordinates to send $X_{0}$ onto
$X$. The second result addresses the rigidity of transverse transforms
within non\textendash{}degenerate algebras. When the web $\Omega:=\left(Z_{0},Y_{0},\hat{Y}\right)$
is non\textendash{}degenerate (\emph{i.e.} the vector fields are generically
pairwise transverse to each\textendash{}others) the convergence will
follow, more precisely:
\begin{thm}
\label{thm:rigid-tr}Let $\lie{Z,Y}$ be a non\textendash{}degenerate
Lie algebra of ratio $\delta$ and consider the associated 3\textendash{}web
$\Omega:=\left(Z_{0},Y_{0,}\hat{Y}\right)$. If $Z$ is non\textendash{}nilpotent
then:
\begin{enumerate}
\item either $\Omega$ is non\textendash{}degenerate and there exists an
analytic conjugacy between $Z$ and $Z_{0}$.
\item either $\hat{Y}$ is not transverse to $Y_{0}$ and, in general, any
formal conjugacy between $Z$ and $Z_{0}$ diverges.
\end{enumerate}
\end{thm}
Case~(1) corresponds to the case where the foliation $\fol Z$ admits
a Godbillon\textendash{}Vey sequence of length $1$, \emph{i.e.} there
exists a closed meromorphic 1\textendash{}form $\tau$ such that $\tau(Z)=0$
or, equivalently,\emph{ }a transverse commuting vector field. This
situation somehow relates to results of~\cite{Cervo} or~\cite{Stolo}.
The case~(2) corresponds to the case where $\fol Z$ admits a Godbillon\textendash{}Vey
sequence of length 2, and no sequence of lesser length in general.
This theorem is a consequence of the study carried out in Sections~\ref{sec:irrat}
and~\ref{sec:res-sad}.

The key to this theorem is the following lemma:
\begin{lem}
\label{lem:tr-homog-rigid}Let $X$ and $Y$ be meromorphic vector
fields on a domain $\mathcal{U}\subset\ww C^{2}$ such that $X\trs Y$,
and $N$ be a formal power series based at a point $p\in\mathcal{U}$.
Assume that $X\ddt N$ and $Y\ddt N$ are meromorphic on $\mathcal{U}$.
Then $N$ is a convergent power series.\end{lem}
\begin{proof}
Let $\left(\tau_{X},\tau_{Y}\right)$ be the dual basis of $\left(X,Y\right)$.
Then $\dd N=\left(X\ddt N\right)\tau_{X}+\left(Y\ddt N\right)\tau_{Y}$
is meromorphic on $\mathcal{U}$. Hence $N$ is a convergent power
series at $p$.
\end{proof}
Surprisingly enough this elementary lemma really is the point of the
result of transverse rigidity we present here.

\subsection{Galois\textendash{}Malgrange groupoid of meromorphic vector fields
$Z$}

~

The last part of this paper, Section~\ref{sec:galois}, introduces
briefly the definition of the Galois\textendash{}Malgrange groupoid
for \emph{meromorphic} vector fields given in~\cite{Malg,MalgShort}
or~\cite{Casa}. It is particularly a subgroupoid of the sheaf $\aut Z$
of germs of a biholomorphic symmetry of $Z$ at points near $\left(0,0\right)$
(perhaps outside an analytic hypersurface). We identify the groupoid
by considering all transverse structures of $Z$
\begin{eqnarray*}
\tx{TS}\left(Z\right) & := & \left\{ \left(\delta,Y\right)\,:\, Y\trs Z,\,\left[Z,Y\right]=\delta Y\right\} \,.
\end{eqnarray*}
The set of all possible such $\delta$ is a $\ww Z$\textendash{}lattice
which classifies most of the dynamical type of $Z$ (Section~\ref{sec:delta-lattice}).
\begin{namedthm}[Integrability Theorem]
Assume $Z$ is a germ of a meromorphic vector field which admits
a non\textendash{}degenerate affine Lie algebra. For $\left(\delta,Y\right)\in\tx{TS}\left(Z\right)$
call $\aut{\lie{Z,Y}}$ the sheaf of germs of biholomorphic symmetries
of the Lie Algebra $\lie{Z,Y}$. The Galois\textendash{}Malgrange
groupoid $\gal Z$ of $Z$ is then given by
\begin{eqnarray*}
\gal Z & = & \aut Z\cap\bigcap_{\left(\delta,Y\right)\in\tx{TS}\left(Z\right)}\aut{\lie{Z,Y}}\,.
\end{eqnarray*}
\end{namedthm}
\begin{rem*}
We also give explicit equations of $\gal Z$ in canonical forms, expressed
as linear differential equations in the Lie derivatives $Z\cdot$
and $Y\cdot$ acting on the sheaf of germs of a holomorphic function.
\end{rem*}

\subsection{Notations}

~

Let us introduce some notations and conventions.
\begin{itemize}
\item $\ww N:=\left\{ 1,2,\ldots\right\} =\ww Z_{>0}$.
\item The spaces $\germ{x_{1},\cdots,x_{n}}$ and $\mero{x_{1},\cdots,x_{n}}$
stand respectively for the spaces of germs at $\left(0,\cdots,0\right)\in\ww C^{n}$
of a holomorphic function and of a meromorphic function.
\item We write $W\ddt F$ for the action of a vector field $W$ as a derivation
on the $\ww C$\textendash{}algebra $\frml{x,y}$ of formal power
series $F$ at $\left(0,0\right)$. Its action is extended component\textendash{}wise
to vector functions.
\item We write $\flw Wt{}$ as the 1\textendash{}parameter (pseudo\textendash{})group
associated to a vector field $W$ (its flow). In the sequel the notation
$\flw WF{}$, where $F$ is a function (or a formal power series),
will stand for the (formal) map $\left(x,y\right)\mapsto\flw W{F\left(x,y\right)}{\left(x,y\right)}$.
\end{itemize}
\tableofcontents{}

\section{\label{sec:Prelim}Preliminary results}

In this section $Z$ is a germ of a holomorphic vector field fear
$\left(0,0\right)$. If $Z\left(0,0\right)=0$ we define $\left\{ \lambda_{1},\lambda_{2}\right\} $
the spectrum of the linear part of $Z$ at $\left(0,0\right)$. When
$Z$ is not nilpotent we label the eigenvalues in such a way that
$\lambda_{2}\neq0$ and set $\lambda:=\nf{\lambda_{1}}{\lambda_{2}}$.
The vector field $Z$ admits a \textbf{reduced} singularity at $\left(0,0\right)$
when $\lambda\notin\ww Q_{>0}$.

\subsection{\label{sec:mero-first-integ-sepx}Separatrices and first\textendash{}integrals}

~

Let us recall briefly basic facts about separatrices.
\begin{defn}
\label{def:separatrix}A \textbf{separatrix} of $Z$ at $\left(0,0\right)$
is a curve $\left\{ f=0\right\} \ni\left(0,0\right)$ where $f$ is
a germ of a (non\textendash{}constant) holomorphic function at $\left(0,0\right)$
such that there exists a germ of a holomorphic function $K$ (called
the cofactor) satisfying 
\begin{eqnarray*}
Z\ddt f & = & Kf\,.
\end{eqnarray*}
By abuse of language we may also say that $f$ is a separatrix of
$Z$.\end{defn}
\begin{itemize}
\item Equivalently this means the curve $\left\{ f=0\right\} $ is tangent
to $Z$, and therefore the adherence of a finite union of leaves of
$\fol Z$.
\item A classical result of \noun{C.~Camacho} and \noun{P.~Sad} asserts
that every vector field of $\left(\ww C^{2},0\right)$ admits at least
one separatrix~\cite{CamaSad}.
\item The level curves of any non\textendash{}constant first\textendash{}integral
$h$ of $Z$ are separatrices, since $Z\ddt h=0$.
\item In the case of a non\textendash{}nilpotent singularity, if $\lambda$
is not a rational number or $\fol Z$ is not formally linearizable
then $Z$ admits no non\textendash{}constant (formal or convergent)
meromorphic first\textendash{}integral. Hence if $Z\ddt a=0$ for
$a$ meromorphic then $a\in\ww C$.\end{itemize}
\begin{lem}
\label{lem:d-basic}Let $Z$ be a germ of a singular vector field
at $\left(0,0\right)$. Then any non\textendash{}zero germ of a meromorphic
function $F$ at $\left(0,0\right)$ such that $Z\cdot F=KF$ with
$K$ holomorphic is of the form 
\begin{eqnarray*}
F & = & \beta\prod_{j}f_{j}^{\alpha_{j}}\,\,\,\,\,\,\,\,,\,\,\alpha_{j}\in\ww Z
\end{eqnarray*}
where $\beta\neq0$ is some germ of a holomorphic function at $\left(0,0\right)$
which is not a separatrix of $Z$ and the $f_{j}$'s form a finite
family of distinct irreducible branches of a separatrix of $Z$.\end{lem}
\begin{proof}
Write $F=\frac{p}{q}$ with $p,\, q$ coprime germs at $\left(0,0\right)$
of a holomorphic function. If $p$ or $q$ does not vanish at $\left(0,0\right)$
the result is trivial. In the other case from 
\begin{eqnarray*}
qZ\cdot p & = & \left(Kq+Z\cdot q\right)p
\end{eqnarray*}
we deduce that $Z\cdot p$ vanishes along the curve $\left\{ p=0\right\} $
and $p$ is a separatrix. The conclusion follows.\end{proof}
\begin{lem}
\label{lem:poles_Y}Let $\lie{Z,Y}$ be a Lie algebra of ratio $\delta$
with $Y$ not holomorphic. The polar locus $\left\{ \frak{p}=0\right\} $
of $Y$ is a separatrix of $Z$.\end{lem}
\begin{proof}
The proof is the same as the previous lemma's, this time using the
equality
\begin{eqnarray}
\frak{p}\left[Z,W\right] & = & \left(\delta\frak{p}+Z\cdot\frak{p}\right)W\label{eq:commut_holo}
\end{eqnarray}
where $W=\frak{p}Y$ for coprime and holomorphic $\frak{p}$, $W$.
\end{proof}

\subsection{Non\textendash{}degenerate algebras}

~

We assume here that the locus of tangency $\left\{ \det\left(Z,Y\right)=0\right\} $
is of codimension at least 1, that is $Z\trs Y$. Let us first show
that the foliation $\fol Z$ admits a Godbillon\textendash{}Vey sequence
of length $1$ or 2:
\begin{lem}
\label{lem:godbillon}Denote by $\left(\tau_{Z},\tau_{Y}\right)$
the dual basis of $\left(Z,Y\right)$. Then $\dd{\tau_{Z}}=0$ and
$\dd{\tau_{Y}}=\delta\tau_{Z}\wedge\tau_{Y}$.\end{lem}
\begin{proof}
We use the formula
\begin{eqnarray*}
\dd{\tau}(Z,Y) & = & Z\ddt\tau(Y)-Y\ddt\tau(Z)-\tau\left(\left[Z,Y\right]\right)
\end{eqnarray*}
to obtain:
\begin{eqnarray*}
\dd{\tau_{Z}}(Z,Y) & = & Z\ddt0-Y\ddt1-\delta\tau_{Z}(Y)\\
 & = & 0\,.
\end{eqnarray*}
Using the formula once more we conclude: 
\begin{eqnarray*}
\dd{\tau_{Y}}(Z,Y) & = & -\delta\\
 & = & \delta\tau_{Y}\wedge\tau_{Z}\left(Z,Y\right)\,.
\end{eqnarray*}

\end{proof}
The polar locus of $\tau_{Z}$ and $\tau_{Y}$ is contained in the
union of the polar locus of $Y$ and the locus of tangency $\left\{ \det\left(Z,Y\right)=0\right\} $
because of Cramer's formula. Taking Lemma~\ref{lem:poles_Y} into
account we obtain the
\begin{lem}
\label{lem:pole-is-sep}The locus of tangency $\left\{ \det\left(Z,Y\right)=0\right\} $
is a union of separatrices of $Z$. Consequently the polar locus of
$\left(\tau_{Z},\tau_{Y}\right)$ also is.\end{lem}
\begin{proof}
Write $Z=A\pp x+B\pp y$ and $Y=C\pp x+D\pp y$ so that $\Delta:=AD-BC=\det\left(Z,Y\right)$.
Because $\left[Z,Y\right]=\delta Y$ we obtain immediately that
\begin{eqnarray*}
Z\ddt\Delta & = & \left(\delta+\pp xA+\pp yB\right)\Delta.
\end{eqnarray*}

\end{proof}

\subsection{Changes of coordinates}

~

We will mostly use two kinds of changes of coordinates to normalize
a non\textendash{}degenerate Lie algebra $\lie{Z,Y}$ of ratio $\delta$,
formed with two formal power series $T,\, N\in\frml{x,y}$ with $T(0,0)=N(0,0)=0$. 
\begin{itemize}
\item The first type of conjugacy will be called \textbf{tangential} since
it is constructed with the flow $\flw Zt{}$ of $Z$:
\begin{eqnarray*}
\mathcal{T}(x,y) & := & \flw Z{T(x,y)}{(x,y)}\,\,.
\end{eqnarray*}

\item The second type, obtained with the flow of the vector field $Y_{0}$,
is called \textbf{transversal}:
\begin{eqnarray*}
\mathcal{N}(x,y) & := & \flw Y{N(x,y)}{(x,y)}.
\end{eqnarray*}

\end{itemize}
Since the flows $\flw Yt{}$ and $\flw Zt{}$ are convergent power
series, these changes of coordinates are convergent if $T$ and $N$
are. Let us explain how they act by conjugacy on $\left(Z,Y\right)$.
\begin{prop}
\label{pro:chg_coord} Assume that $\left[Z,Y\right]=DY$ with $Y\ddt D=0$
and $Y$ holomorphic, not necessarily transverse.
\begin{enumerate}
\item Let $W$ be a holomorphic vector field near $\left(0,0\right)$ and
$F\in\frml{x,y}$ such that $F\left(0,0\right)=0$. Then $\flw WF{}$
is a formal change of coordinates if, and only if, $\left(W\ddt F\right)\left(0,0\right)\neq-1$.
It fixes $\left(0,0\right)$ and when $W$ is singular it is tangent
to the identity. 
\item $\mathcal{N}^{*}Z=Z-\left(Z\ddt N+DN\right)\mathcal{N}^{*}Y$ where
$\mathcal{N}^{*}Y=\frac{1}{1+Y\ddt N}Y$.
\item $\mathcal{T}^{*}Z=\frac{1}{1+Z\ddt T}Z$. If moreover $D\in\ww C$
and if $Z\trs Y$ then $\mathcal{T}^{*}Y=e^{DT}\left(Y-\left(Y\ddt T\right)\mathcal{T}^{*}Z\right)$.
\end{enumerate}
\end{prop}
\begin{rem}
If $W$ is singular at $\left(0,0\right)$ then $\left(W\ddt F\right)(0,0)=0$.\end{rem}
\begin{proof}
~
\begin{enumerate}
\item Define the convergent vector power series at $\left(0,0,0\right)$
\begin{eqnarray*}
g(x,y,t) & := & \flw Wt{(x,y)}
\end{eqnarray*}
which can be written
\begin{eqnarray*}
g(x,y,t) & = & \sum_{n=0}^{+\infty}\frac{t^{n}}{n!}W\ddt^{n}\left(x,y\right)
\end{eqnarray*}
where $W\ddt^{n+1}:=W\ddt W\ddt^{n}$ and $W\ddt^{0}:=Id$. The object
$g(x,y,F(x,y))$ is then a formal power series since $F(0,0)=0$.
When $W$ is singular a simple computation ensures that $\flw WF{}$
is tangent to identity and thus is a formal change of coordinates.
The fact that the Jacobian of $\flw WF{}$ at $\left(0,0\right)$
is $\left(1+\left(W\ddt F\right)\left(0,0\right)\right)$ is proved
in the next point.
\item Consider the vector formal power series
\begin{eqnarray*}
\pi\,:\,\left(x,y\right) & \mapsto & \left(x,y,N(x,y)\right)\,,
\end{eqnarray*}
as well as the convergent power series 
\begin{eqnarray*}
g(x,y,t): & = & \flw Yt{(x,y)}.
\end{eqnarray*}
Writing $\mathcal{N}=g\circ\pi$ we obtain  : 
\begin{eqnarray*}
D\mathcal{N} & = & DgD\pi\\
 & = & \left(Dg\left[\begin{array}{cc}
1 & 0\\
0 & 1\\
\frac{\partial N}{\partial x} & \frac{\partial N}{\partial y}
\end{array}\right]\right)\circ\pi\,.
\end{eqnarray*}
The determinant of this matrix is
\begin{eqnarray*}
\det\left(D\mathcal{N}\right) & = & \det\left(\left[\begin{array}{cc}
\frac{\partial g_{x}}{\partial x} & \frac{\partial g_{y}}{\partial x}\\
\frac{\partial g_{x}}{\partial x} & \frac{\partial g_{y}}{\partial x}
\end{array}\right]\circ\pi\right)+\det\left(\left[\begin{array}{cc}
\frac{\partial g_{x}}{\partial t} & \frac{\partial g_{y}}{\partial t}\\
\frac{\partial g_{x}}{\partial y} & \frac{\partial g_{y}}{\partial y}
\end{array}\right]\circ\pi\right)\frac{\partial N}{\partial x}\\
 &  & +\det\left(\left[\begin{array}{cc}
\frac{\partial g_{x}}{\partial x} & \frac{\partial g_{y}}{\partial x}\\
\frac{\partial g_{x}}{\partial t} & \frac{\partial g_{y}}{\partial t}
\end{array}\right]\circ\pi\right)\frac{\partial N}{\partial y}
\end{eqnarray*}
which we evaluate at $\left(0,0\right)$  :
\begin{eqnarray*}
\det\left(D\mathcal{N}\left(0,0\right)\right) & = & 1+\left(Y\ddt N\right)\left(0,0\right)\,,
\end{eqnarray*}
since $\frac{\partial g}{\partial t}\left(x,y,0\right)=Y\left(x,y\right)$
(component\textendash{}wise) and $g\left(\ddt,\ddt,0\right)=Id$.
\\
Define $\tilde{Z}:=Z+RY$; the conjugacy equation $\mathcal{N}^{*}Z=\tilde{Z}$
writes, component\textendash{}wise, 
\begin{eqnarray}
Z\circ g\circ\pi & = & Z\ddt\left(g\circ\pi\right)+RY\ddt\left(g\circ\pi\right)\nonumber \\
 & = & D\mathcal{N}\left(Z+RY\right)\nonumber \\
 & = & \left(Z\ddt g+\left(Z\ddt N\right)\frac{\partial g}{\partial t}+R\left(Y\ddt g+\left(Y\ddt N\right)\frac{\partial g}{\partial t}\right)\right)\circ\pi\label{eq:conjug}
\end{eqnarray}
Since $g(\ddt,\ddt t)^{*}Y=Y$ we have $\frac{\partial g}{\partial t}=Y\circ g=Y\ddt g$.
On the other hand, 
\begin{eqnarray*}
Z\ddt g & = & \sum_{n\geq0}\frac{t^{n}}{n!}Z\ddt Y\ddt^{n}Id\,.
\end{eqnarray*}
For $Y\ddt D=0$ and $\left[Z,Y\right]=DY$ the equality $Z\ddt Y\ddt^{n}=Y\ddt^{n}Z\ddt+nDY\ddt^{n}$
holds, which further yields  : 
\begin{eqnarray*}
Z\ddt g & = & Z\circ g+tDY\circ g\,.
\end{eqnarray*}
Equation (\ref{eq:conjug}) writes now
\begin{eqnarray*}
Z\circ g\circ\pi & = & Z\circ g\circ\pi+\left(Z\ddt N+DN+R\left(1+Y\ddt N\right)\right)Y\circ g\circ\pi
\end{eqnarray*}
 and is satisfied if, and only if, 
\begin{eqnarray*}
R & = & -\frac{Z\ddt N+DN}{1+Y\ddt N}\,.
\end{eqnarray*}

\item The first statement is actually given by (2) when $Y:=Z$, $D:=0$
and $N:=T$. Write $\mathcal{T}^{*}Y=U\left(Y+RZ\right)$ and $\delta:=D$.
With the corresponding notations, we need to solve
\begin{eqnarray*}
Y\circ g\circ\pi & = & U\left(Y\ddt g+\left(Y\ddt T+R\left(1+Z\ddt T\right)\right)Z\circ g\right)\circ\pi\,.
\end{eqnarray*}
Taking into account that
\begin{eqnarray*}
Y\ddt Z\ddt^{n} & = & \sum_{p\leq n}C_{n}^{p}\left(-\delta\right)^{n-p}Z\ddt^{p}Y\ddt
\end{eqnarray*}
we obtain
\begin{eqnarray*}
Y\ddt g & = & \sum_{0\leq n}\sum_{p\leq n}\frac{t^{n}}{n!}C_{n}^{p}\left(-\delta\right)^{n-p}Z\ddt^{p}Y\ddt Id\\
 & = & \sum_{0\leq p}\left(\sum_{p\leq n}\frac{\left(-\delta t\right)^{n-p}}{\left(n-p\right)!}\right)\frac{t^{p}}{p!}Z\ddt^{p}Y\ddt Id\\
 & = & e^{-\delta t}Y\circ g
\end{eqnarray*}
which yields the result since $Z\trs Y$. 
\end{enumerate}
\end{proof}

\subsection{Affine Lie algebras}

~

We first show that $Z$ cannot be too degenerate. Then we deal with
a non\textendash{}isolated singularity as a showcase situation for
the subtler setting of (quasi\textendash{})resonant singularities.

\bigskip{}

If $Z$ is singular let $Z_{0}$ denote the linear part of $Z$ at
$\left(0,0\right)$. Up to a linear change of variables we can put
$Z_{0}$ under Jordan normal form 
\begin{eqnarray*}
Z_{0}(x,y) & = & \lambda_{1}x\pp x+\lambda_{2}\left(y+\varepsilon x\right)\pp y\,\,,\,\varepsilon\in\left\{ 0,1\right\} \,,
\end{eqnarray*}
with $\varepsilon=0$ if $\lambda_{1}\neq\lambda_{2}$. In all the
paragraph we consider some meromorphic vector field $Y=\frac{1}{\frak{p}}W$
where $W\neq0$ is a holomorphic vector field and $\frak{p}$ a holomorphic
function, coprime with $W$. If $Y$ is actually holomorphic we make
the convention $\frak{p}:=1$. We suppose that $\lie{Z,Y}$ is affine,
which may be degenerate. We recall that formula~\eqref{eq:commut_holo}
and Lemma~\ref{lem:poles_Y} imply the existence of a holomorphic
$K$ with
\begin{eqnarray*}
Z\cdot\frak{p} & = & K\frak{p}\\
\left[Z,W\right] & = & \left(\delta+K\right)W\,.
\end{eqnarray*}

\begin{lem}
\label{lem:degen}The linear part $Z_{0}$ is non\textendash{}nilpotent.\end{lem}
\begin{proof}
Assume $Z_{0}=0$. Looking at the least homogeneous degree on both
sides of $Z\cdot\frak{p}=K\frak{p}$ implies that $K\left(0,0\right)=0$.
On the other hand, noting $W_{0}\neq0$ the least homogeneous part
of $W$:
\begin{eqnarray*}
0=Z_{0}\cdot W_{0}-W_{0}\cdot Z_{0} & = & \delta W_{0}\,.
\end{eqnarray*}
Therefore $\delta=0$.

Assume now $Z_{0}\neq0$ but $\lambda_{1}=\lambda_{2}=0$ with $\varepsilon=1$.
We mention the following immediate lemma:
\begin{lem}
\label{lem:d-basic_homog}Let $\alpha\in\ww C$ and $f\in\pol{}{x,y}$
be a homogeneous polynomial such that $y\frac{\partial f}{\partial x}=\alpha f$.
Then $\alpha f=0$.
\end{lem}
If $K\left(0,0\right)\neq0$ the relation $Z\cdot\frak{p}=K\frak{p}$
reads $Z_{0}\cdot\frak{p}_{0}=y\frac{\partial\frak{p}_{0}}{\partial x}=K\left(0,0\right)\frak{p}_{0}$
for the least homogeneous degree. Since $\frak{p}_{0}\neq0$ this
contradicts the above lemma, so that $K\left(0,0\right)=0$. Now the
relation 
\begin{eqnarray*}
Z_{0}\cdot W_{0}-W_{0}\cdot Z_{0} & = & \delta W_{0}
\end{eqnarray*}
holds and writing $W_{0}:=C\pp x+D\pp y$ we derive $y\pp xC=\delta C$
and $y\pp xD=\delta D+C$. Using again the lemma we deduce $C=D=0$,
a contradiction.\end{proof}
\begin{cor}
\label{cor:non-isolated}Let $Z$ be a holomorphic singular vector
field with a non\textendash{}isolated singularity at $\left(0,0\right)$.
If there exists an affine Lie algebra $\lie{Z,Y}$ of ratio $\delta$,
maybe degenerate, then there exists a local analytic change of coordinates
$\psi$ such that $\psi^{*}Z=\lambda_{2}y\left(1+\mu x^{k}\right)\pp y$
with $\lambda_{2}\neq0$, $k\in\ww N_{}$ and $\mu\in\ww C$. If moreover
$Y\trs Z$ then $\mu=0$ and we have, for a unique $m\in\ww Z$ and
meromorphic germs $a,\, b\in\mero x$,~$b\neq0$:
\begin{eqnarray*}
\psi^{*}Y & = & y^{m}\left(a\left(x\right)\psi^{*}Z+b\left(x\right)\pp x\right)\\
\delta & = & m\lambda_{2}\,.
\end{eqnarray*}
\end{cor}
\begin{proof}
The previous lemma states that $Z_{0}\neq0$ so if $\left\{ s=0\right\} $
is the singular locus of $Z$ then the homogeneous valuation of $s$
at $\left(0,0\right)$ is $1$ and $Z=sX$ with $X\left(0,0\right)\neq0$.
Up to change the local coordinates we may assume $X=\pp y$ and $s\left(x,y\right)=\lambda_{2}y+\cdots$
with $\lambda_{2}\neq0$. According to the implicit function theorem
we can locally write
\begin{eqnarray*}
\left\{ s\left(x,y\right)=0\right\}  & = & \left\{ y=u\left(x\right))\right\} 
\end{eqnarray*}
 for some holomorphic $u$, so that $\left(x,y\right)\mapsto\left(x,y-u\left(x\right)\right)$
provides a conjugacy between $Z$ and $\tilde{Z}\left(x,y\right):=\lambda_{2}\tilde{s}(x,y)y\pp y$
with $\tilde{s}(0,0)=1$. Now, according to Proposition~\ref{pro:chg_coord},
conjugating $\hat{Z}\left(x,y\right):=\lambda_{2}\tilde{s}(x,0)y\pp y$
to $\tilde{Z}$ through a tangential change of coordinates $\mathcal{T}:=\flw{\hat{Z}}T{}$
is equivalent to solving
\begin{eqnarray*}
\lambda_{2}y\frac{\partial T}{\partial y}(x,y) & = & \frac{1}{\tilde{s}(x,y)}-\frac{1}{\tilde{s}(x,0)}\,.
\end{eqnarray*}
A holomorphic solution $T$ always exists. 

Finally we conjugate some $\lambda_{2}\left(1+\mu x^{k}\right)y\pp y$
to $\lambda_{2}\tilde{s}(x,0)y\pp y$ by using a transverse change
of coordinates $\mathcal{N}(x,y):=\flw{x\pp x}{N(x)}{(x,y)}$. This
change of coordinates is an isotropy of $\lambda_{2}y\pp y$ so that
we only need to solve 
\begin{eqnarray*}
\tilde{s}(x,0) & = & 1+\mu x^{k}\exp\left(pN(x)\right).
\end{eqnarray*}
Either $\tilde{s}\left(x,0\right)$ is constant and we set $\mu:=0$,
or $\tilde{s}(x,0)=1+\mu\left(x^{k}+\cdots\right)$ with $\mu\neq0$.
In any case we obtain a holomorphic solution $N$.

Assume that the coordinates are changed accordingly: $Z\left(x,y\right)=\lambda_{2}y\left(1+\mu x^{k}\right)\pp y$
and write $Y=AZ+B\pp x$ for meromorphic germs $A,\, B$ with $B\neq0$.
The relation $\left[Z,Y\right]=\delta Y$ becomes 
\begin{eqnarray*}
\begin{cases}
Z\cdot A & =\delta A+\frac{k\mu x^{k-1}}{1+\mu x^{k}}B\\
Z\cdot B & =\delta B
\end{cases} &  & \,.
\end{eqnarray*}
If $\delta=0$ the second equation yields $B\in\mero x$ and taking
$y:=0$ in the first one yields $\mu=0$, then $A\in\mero x$. On
the contrary when $\delta\neq0$ we have 
\begin{eqnarray*}
B\left(x,y\right) & = & b\left(x\right)y^{\frac{\delta}{\lambda_{2}\left(1+\mu x^{k}\right)}}
\end{eqnarray*}
so that the meromorphy of $B$ imposes $\mu=0$ and $m:=\nf{\delta}{\lambda}_{2}\in\ww Z$.
Plugging this relation into the first equation finally gives $A\left(x,y\right)=a\left(x\right)y^{1+m}$.
\end{proof}

\section{\label{sec:irrat}Quasi\textendash{}resonant vector fields}

We consider here 
\begin{eqnarray*}
Z\left(x,y\right) & = & \left(\lambda_{1}x+\ldots\right)\pp x+\left(\lambda_{2}y+\ldots\right)\pp y
\end{eqnarray*}
where «$\ldots$» stands for terms of homogeneous degree greater than
$1$ and $\lambda:=\lambda_{1}/\lambda_{2}\in\ww R_{<0}\backslash\ww Q$.
We denote the linear part of $Z$ by $Z_{0}$. In this section we
aim to prove the following result.
\begin{thm}
\label{thm:irrat-th}Assume that $\lambda\in\ww R_{<0}\backslash\ww Q$.
Any non degenerate Lie algebra $\lie{Z,Y}$ of ratio $\delta$ is
analytically conjugate to some $\lie{Z_{0},\tilde{Y}}$ where $Z_{0}:=\lambda_{1}x\pp x+\lambda_{2}y\pp y$
is the linear part of $Z$ and $\tilde{Y}$ is as follows. There exist
a unique $\left(n,m\right)\in\ww Z^{2}$ and $\left(d,c\right)\in\ww C$
such that
\begin{eqnarray*}
\tilde{Y}\left(x,y\right) & = & x^{n}y^{m}\left(dx\pp x+cy\pp y\right)\,.
\end{eqnarray*}
 The ratio of $\lie{Z,Y}$ is given by
\begin{eqnarray*}
\delta & = & n\lambda_{1}+m\lambda_{2}\,.
\end{eqnarray*}

\end{thm}

\subsection{Preparation}

~

It is well known that the foliation induced by $Z$ admits two and
only two reduced separatrices, say $\mathcal{S}_{x}$ and $\mathcal{S}_{y}$
(see~\cite{CamaSad}). They are tangent to the eigenvectors of $Z_{0}$,
respectively $\left[\begin{array}{c}
0\\
1
\end{array}\right]$ and $\left[\begin{array}{c}
1\\
0
\end{array}\right]$. Hence for a sufficiently small open polydisc $V$ centered at $\left(0,0\right)$
there exist two holomorphic functions $f_{x}\in\mathcal{O}\left(V\cap\left\{ x=0\right\} \right)$
and $f_{y}\in\mathcal{O}\left(V\cap\left\{ y=0\right\} \right)$ such
that $f_{x}\left(0\right)=f_{y}\left(0\right)=0$ and  :
\begin{eqnarray*}
\mathcal{S}_{x}\cap V & = & \left\{ x=f_{x}(y)\right\} \,,\\
\mathcal{S}_{y}\cap V & = & \left\{ y=f_{y}(x)\right\} \,.
\end{eqnarray*}
 The analytic change of coordinates 
\begin{eqnarray*}
\psi\,:\, V & \to & \ww C^{2}\\
\left(x,y\right) & \mapsto & \left(x-f_{x}\left(y\right),y-f_{y}\left(x\right)\right)
\end{eqnarray*}
brings $Z$ into the vector field
\begin{eqnarray*}
\psi^{*}Z & = & \lambda_{1}x\left(1+A\right)\pp x+\lambda_{2}y\left(1+B\right)\pp y
\end{eqnarray*}
with $A\left(0,0\right)=B\left(0,0\right)=0$. As of now we consider
vector fields $Z$ written in the form  :
\begin{eqnarray*}
Z & = & UX\\
X & = & Z_{0}+Ry\pp y
\end{eqnarray*}
where 
\begin{eqnarray*}
Z_{0} & := & \lambda_{1}x\pp x+\lambda_{2}y\pp y\\
U & := & \left(1+A\right)\\
R & := & \lambda_{2}\left(\frac{1+B}{1+A}-1\right)\,.
\end{eqnarray*}
Notice that $U\left(0,0\right)=1$ and $R\left(0,0\right)=0$. 
\begin{rem}
\label{rem:irrat-equation-in-sep}The only separatrices of $Z$ are
now the branches of $\left\{ xy=0\right\} $. 
\end{rem}

\subsection{\label{sub:irrat-Z-FNF}Formal normalization}

~

We are going to show that $Z$ is formally conjugate to $Z_{0}$.
Although this fact is classical we include a short proof here in order
to introduce an ingredient that will play an important role in the
following. According to Proposition~\ref{pro:chg_coord} we want
to find formal power series $T,\, N\in\frml{x,y}$ such that $T\left(0,0\right)=N\left(0,0\right)=0$
and
\begin{eqnarray*}
X\cdot T & = & \frac{1}{U}-1\\
X\ddt N & = & -R
\end{eqnarray*}
since then $\psi_{T}:=\flw XT{}$ and $\psi_{N}:=\flw{y\pp y}N{}$
satisfies
\begin{eqnarray*}
\psi_{N}^{*}Z_{0} & = & X\\
\psi_{T}^{*}X & = & UX
\end{eqnarray*}
(notice that $\left[X_{0},y\pp y\right]=0$ so that $D=0$ in the
above\textendash{}mentioned proposition).
\begin{lem}
\label{lem:irrat-eq-h-form}Let $G\in\frml{x,y}$ be given. There
exists $F\in\frml{x,y}$ such that $X\ddt F=G$ if, and only if, $G\left(0,0\right)=0$.
The power series $F-F\left(0,0\right)$ is unique.\end{lem}
\begin{proof}
Write $G\left(x,y\right)=\sum_{a,b}g_{a,b}x^{a}y^{b}$ and $F\left(x,y\right)=\sum_{a,b}f_{a,b}x^{a}y^{b}$.
Then
\begin{eqnarray*}
\left(a\lambda_{1}+b\lambda_{2}\right)f_{a,b}+o(a,b) & = & g_{a,b}\,,
\end{eqnarray*}
where $o\left(a,b\right)$ stands for terms involving $f_{c,d}$ for
$c+d<a+b$ only. Since $a\lambda_{1}+b\lambda_{2}\neq0$ if $\left(a,b\right)\neq\left(0,0\right)$
the only constraint is $g_{0,0}=0$; moreover $F$ is unique up to
the choice of $f_{0,0}$.\end{proof}
\begin{cor}
\label{cor:irrat-Z-FNF}The vector field $Z$ is formally conjugate
to its linear part $Z_{0}$.
\end{cor}

\subsection{\label{sub:irrat-form-aff-lie}Lie algebras of $Z_{0}$}

~

Starting with $\lie{Z,Y}$ we have now to consider $\hat{\psi}^{*}\lie{Z,Y}=\lie{Z_{0},\hat{Y}}$
with $\hat{Y}:=\hat{\psi}^{*}Y$ a formal meromorphic vector field.
An important part of the rigidity property we will study below needs
to know that in fact $\hat{Y}$ is a meromorphic vector field. As
a matter of fact we will show that it is essentially unique.
\begin{prop}
\label{pro:irrat-form-aff-lie}Let $\hat{Y}\neq0$ be a formal meromorphic
vector field satisfying $\left[Z_{0},\hat{Y}\right]=\delta\hat{Y}$
with $\delta\in\ww C$. There exists two unique pairs $\left(n,m\right)\in\ww Z^{2}$
and $\left(c,d\right)\in\ww C^{2}$ such that 
\begin{eqnarray*}
\delta & = & n\lambda_{1}+m\lambda_{2}
\end{eqnarray*}
and 
\begin{eqnarray*}
\hat{Y}\left(x,y\right) & = & x^{n}y^{m}\left(dx\pp x+cy\pp y\right)\,.
\end{eqnarray*}
\end{prop}
\begin{proof}
Write $\hat{Y}=D\pp x+C\pp y$ for some $C,\, D\in\fmero{x,y}$. Then
 :
\begin{eqnarray*}
Z_{0}\ddt D-\lambda_{1}D & = & \delta D\\
Z_{0}\ddt C-\lambda_{2}C & = & \delta C\,.
\end{eqnarray*}
According to Lemma~\ref{lem:d-basic} we can write 
\begin{eqnarray*}
D\left(x,y\right) & = & x^{N}y^{M}\gamma\left(x,y\right)\\
C\left(x,y\right) & = & x^{N'}y^{M'}\tilde{\gamma}\left(x,y\right)
\end{eqnarray*}
with either $\gamma=0$ or $\gamma\left(0,0\right)\neq0$ (and the
same for $\tilde{\gamma}$). Assume that $\gamma\neq0$ (so that $\log\gamma\in\frml{x,y}$);
we immediately obtain that
\begin{eqnarray*}
\left(\lambda_{1}N+\lambda_{2}M\right)D+x^{N}y^{M}Z_{0}\ddt\gamma & = & \left(\delta+\lambda_{1}\right)D\,,
\end{eqnarray*}
which we divide by $D$  :
\begin{eqnarray*}
Z_{0}\ddt\log\gamma & = & \left(\delta+\lambda_{1}\left(1-N\right)-\lambda_{2}M\right)\,.
\end{eqnarray*}
We now apply Lemma~\ref{lem:irrat-eq-h-form}: on the one hand 
\begin{eqnarray*}
\lambda_{1}\left(N-1\right)+\lambda_{2}M & = & \delta
\end{eqnarray*}
which determines $\left(N,M\right)$ uniquely in terms of $\left(\lambda_{1},\lambda_{2},\delta\right)$,
while on the other hand
\begin{eqnarray*}
\log\gamma & \in & \ww C\,.
\end{eqnarray*}
The same argument applies for $C$ yielding $\lambda_{1}N'+\lambda_{2}\left(M'-1\right)=\delta$,
so we conclude that:
\begin{eqnarray*}
D\left(x,y\right) & = & dx^{N-1}y^{M}\\
C\left(x,y\right) & = & cx^{N'}y^{M'-1}
\end{eqnarray*}
 with $c,d\in\ww C$. Set now $n:=N-1=N'$ and $m:=M=M'-1$.
\end{proof}

\subsection{\label{sub:irrat-rigid}Rigidification}

~

We intend to show that the power series $N$ and $T$ constructed
in Section~\ref{sub:irrat-Z-FNF} are actually convergent power series
provided the existence of a non\textendash{}degenerate Lie algebra
$\lie{Z,Y}$ of ratio $\delta$. Notice that under some arithmetic
condition on $\lambda$ (see~\cite{Ricard}), called ``Brjuno condition''
and satisfied for $\lambda$ belonging to a set of full measure in
$\ww R$, it is not necessary to assume the existence of such a $Y$
as every $Z$ is analytically conjugate to $Z_{0}$. We are not concerned
here with such considerations.

\subsubsection{\label{sub:irrat-tang-rigid}Tangential rigidity}

~

According to Lemma~\ref{lem:godbillon} there exists a closed meromorphic
1\textendash{}form $\tau_{Z}$ such that $\tau_{Z}\left(Z\right)=1$
(a closed time\textendash{}form). On the other hand the closed 1\textendash{}form
$\tau_{X}:=\frac{\dd x}{\lambda_{1}x}$ is a time\textendash{}form
for $X$. Hence 
\begin{eqnarray*}
\tau & := & \tau_{Z}-\tau_{X}
\end{eqnarray*}
 is a closed meromorphic 1\textendash{}form. Because the poles of
$\tau_{Z}$ and $\tau_{X}$ are included in $\left\{ xy=0\right\} $
(Remark~\ref{rem:irrat-equation-in-sep} and Lemma~\ref{lem:pole-is-sep})
we necessarily have that
\begin{eqnarray*}
\tau & = & \dd{\left(\frac{A\left(x,y\right)}{x^{N}y^{M}}+\alpha\log x+\beta\log y\right)}\\
 & = & \dd F
\end{eqnarray*}
with $\left(\alpha,\beta,N,M\right)\in\ww C^{2}\times\ww Z_{\geq0}^{2}$
and $A$ holomorphic near $\left(0,0\right)$. Observe that
\begin{eqnarray*}
X\ddt F & = & \tau\left(X\right)=\tau_{Z}\left(X\right)-1\\
 & = & \frac{1}{U}-1
\end{eqnarray*}
so we have built a meromorphic, multivalued solution to the equation
of tangential normalization.
\begin{lem}
\label{lem:irrat-tang-rigid}Let $\mathcal{A}$ be the linear space
over $\ww C$ of all formal expressions 
\[
\frac{A\left(x,y\right)}{x^{N}y^{M}}+\alpha\log x+\beta\log y
\]
with $A\in\frml{x,y}$ and $\left(\alpha,\beta,N,M\right)\in\ww C^{2}\times\ww Z_{\geq0}^{2}$.
Assume that $F\in\mathcal{A}$ is such that $X\ddt F=G\in\frml{x,y}$.
Then there exists $\hat{F}\in\frml{x,y}$ and a unique $\alpha\in\ww C$
such that
\begin{eqnarray*}
F\left(x,y\right) & = & \hat{F}\left(x,y\right)+G\left(0,0\right)\log y+\alpha\left(\log x-\lambda\log y\right)\,.
\end{eqnarray*}
The power series $\hat{F}-\hat{F}\left(0,0\right)$ is unique. On
the other hand every $G\in\frml{x,y}$ writes $X\ddt F$ for some
$F\in\mathcal{A}$ with $\alpha=0$. 
\end{lem}
Before giving the proof of this lemma we conclude our discussion about
tangential conjugacy. Since $\tau=\dd F$ and $F,\, T\in\mathcal{A}$
with $X\ddt F=X\ddt T\in\frml{x,y}$ we obtain
\begin{eqnarray*}
T & = & \hat{F}-\hat{F}\left(0,0\right)\,.
\end{eqnarray*}
Hence $T$ is a convergent power series since $F$ is a convergent
object. What we have in fact proved is the following:
\begin{prop}
\label{pro:irrat-tang-rigid} Assume that $Z$ admits a closed time\textendash{}form
whose polar locus is included in $\left\{ xy=0\right\} $. Then $T$
is a convergent power series.\end{prop}
\begin{cor}
\label{cor:irrat-tang-rigid}There exists a closed time\textendash{}form
for $Z$ if, and only if, $Z$ is analytically conjugate to $X$ by
a tangential change of coordinates.
\end{cor}
This result proves Theorem~\ref{thm:rigid-tang} for quasi\textendash{}resonant
vector fields (setting $Q:=1$).
\begin{proof}
Because $\dd{\tau_{X}}=0$ the existence of an analytical conjugacy
gives rise to the closed time\textendash{}form $\tau:=\psi^{*}\tau_{X}$
for $Z$. Assume conversely that there exists some $\tau$ with $\tau\left(Z\right)=1$
and $\dd{\tau}=0$. To apply the previous proposition we only need
to show that the polar locus of $\tau$ is empty or a separatrix of
$Z$. Since $Z=UX=A\pp x+B\pp y$ we can write
\begin{eqnarray*}
\tau & = & \frac{1}{U}\tau_{X}+\frac{f}{\frak{p}}\left(A\dd y-B\dd x\right)
\end{eqnarray*}
with $f,\,\frak{p}$ coprime and holomorphic. The case $\frak{p}\left(0,0\right)\neq0$
is trivial so we may as well assume $\left\{ \frak{p}=0\right\} \neq\emptyset$.
Taking into account that $\tau$ and $\tau_{X}$ are closed we derive
the relations
\begin{eqnarray*}
0=\dd{\tau} & = & \left(Z\cdot\frac{f}{\frak{p}}+\frac{f}{\frak{p}}\tx{div}Z\right)\dd x\wedge\dd y-\frac{\dd U}{U^{2}}\wedge\tau_{X}\\
0 & = & \frak{p}Z\cdot f-fZ\cdot\frak{p}+\frak{p}f\tx{div}Z-\frac{\frak{p}^{2}}{\lambda_{1}xU^{2}}\frac{\partial U}{\partial y}\,.
\end{eqnarray*}
From this we deduce that $Z\cdot\frak{p}$ vanishes along the curve
$\left\{ \frak{p}=0\right\} $, meaning that $\frak{p}$ is a separatrix
of $Z$.
\end{proof}
We now give the proof of Lemma~\ref{lem:irrat-tang-rigid}.
\begin{proof}
Write 
\begin{eqnarray*}
F\left(x,y\right) & = & \frac{A\left(x,y\right)}{x^{N}y^{M}}+\alpha\left(\log x-\lambda\log y\right)+\beta\log y
\end{eqnarray*}
with $\left(\alpha,\beta\right)\in\ww C^{2}$ and 
\begin{eqnarray*}
A\left(x,y\right) & = & \sum_{a,b\geq0}f_{a,b}x^{a}y^{b}\,.
\end{eqnarray*}
The equality $X\ddt F=G$ can be rewritten in the following way  :
\begin{eqnarray}
X\ddt A-A\left(N\lambda\left(1+R\right)+M\right) & = & x^{N}y^{M}\left(G-\beta-\alpha R\right)\,.\label{eq:irrat-eq-h-mero}
\end{eqnarray}
The coefficient $h_{a,b}$ of $x^{a}y^{b}$ in the previous equality
is given by
\begin{eqnarray}
\left(\left(a-N\right)\lambda_{1}+\left(b-M\right)\lambda_{2}\right)f_{a,b}-N\lambda_{2}\sum_{c=0}^{a}\sum_{d=0}^{b}f_{c,d}r_{a-c,b-d} & = & h_{a,b}\label{eq:irrat-eq-h-coef}
\end{eqnarray}
where $R\left(x,y\right)=\sum_{a+b>0}r_{a,b}x^{a}y^{b}$. Obviously
$h_{a,b}=0$ if $a<N$ or $b<M$. Reasoning by induction on $a<N$
we show that $f_{a,b}=0$ for all $b$. Indeed assume that $f_{c,d}=0$
for all $c\leq a$ and $d<b$ (or $d\in\ww Z_{\geq0}$ when $c<a$).
Because $r_{0,0}=0$ and $\lambda\notin\ww Q$ the expression~(\ref{eq:irrat-eq-h-coef})
implies that $f_{a,b}=0$. The same argument proves also that $f_{a,b}=0$
if $b<M$ for all $a$. As a consequence there exists some $\tilde{A}\in\frml{x,y}$
such that
\begin{eqnarray*}
A & = & x^{N}y^{M}\tilde{A}\,.
\end{eqnarray*}

Equation~(\ref{eq:irrat-eq-h-mero}) divided by $x^{N}y^{M}$ and
evaluated at $\left(0,0\right)$ yields $\beta=G\left(0,0\right)$.
The remaining of the claim is now clear.
\end{proof}

\subsubsection{\label{sub:irrat-norm-rigid}Transversal rigidity}

~

Corollary~\ref{cor:irrat-tang-rigid} together with Lemma~\ref{lem:godbillon}
allow us to restrict our study to Lie algebras $\mathcal{L}\left(Z,Y\right)$
of ratio $\delta$  with
\begin{eqnarray*}
Z & =X= & Z_{0}+RY_{0}
\end{eqnarray*}
where $Y_{0}=y\pp y$. We know that the formal power series $N$ satisfies
\begin{eqnarray*}
X\ddt N & = & -R\,.
\end{eqnarray*}
In order to apply Lemma~\ref{lem:tr-homog-rigid} we only need to
know that $Y_{0}\ddt N$ is a convergent power series. When this is
the case we deduce that $N$ is convergent because $X\trs Y_{0}$.
\begin{lem}
\label{lem:irrat-norm-rigid}Assume that $\hat{Y}\trs Y_{0}$. Then
$Y_{0}\ddt N$ is a convergent power series.\end{lem}
\begin{proof}
Write $Y_{0}=aZ_{0}+b\hat{Y}$ for some meromorphic functions $a\neq0$
and $b$. The discussion in Section~\ref{sec:mero-first-integ-sepx}
yields $a\in\ww C_{\neq0}$. Since $\mathcal{N}=\flw{Y_{0}}N{}$ and
$\mathcal{N}^{*}\left(Z_{0},\hat{Y}\right)=\left(Z,Y\right)$ we deduce
from Proposition~\ref{pro:chg_coord} that
\begin{eqnarray*}
\mathcal{N}^{*}Y_{0} & = & \frac{1}{1+Y_{0}\ddt N}Y_{0}\\
 & = & aZ+b\circ\mathcal{N}Y\,.
\end{eqnarray*}
As a consequence $a\left(1+Y_{0}\ddt N\right)$ is convergent because
$Z\trs Y$. Since $a\neq0$ it follows that $Y_{0}\ddt N$ is convergent.\end{proof}
\begin{cor}
\label{cor:irrat-norm-rigid}The pair $\left(X,Y\right)$ is analytically
conjugate to $\left(Z_{0},\hat{Y}\right)$.\end{cor}
\begin{proof}
If $\hat{Y}\trs Y_{0}$ then the previous lemma proves the claim.
Conversely assume that 
\begin{eqnarray*}
\hat{Y} & = & cx^{n-1}y^{m}\pp y\,\,\,,\, c\in\ww C.
\end{eqnarray*}
Let $Y_{1}:=x\pp x$ which is transverse to $\hat{Y}$ and commutes
to $Z_{0}$. By repeating the construction at the beginning of the
section, we can write
\begin{eqnarray*}
Z & = & \tilde{U}\tilde{X}\\
\tilde{X} & = & Z_{0}+\lambda_{1}\tilde{R}Y_{1}
\end{eqnarray*}
with $\tilde{R}$ holomorphic and $\tilde{R}\left(0,0\right)=0$.
The conclusion follows.
\end{proof}

\section{\label{sec:res-sad}Resonant singularities}

The spirit of this section is the same as the previous one's. Here
we assume that $\lambda:=\lambda_{1}/\lambda_{2}=-p/q\in\ww Q_{\leq0}$
with $p\wedge q=1$ if $p\neq0$ and $q:=1$ if $p=0$. For $\left(k,\mu\right)\in\ww N_{}\times\ww C$
set 
\begin{eqnarray*}
u\left(x,y\right) & := & x^{q}y^{p}\\
W_{0}\left(x,y\right) & := & \lambda_{1}x\pp x+\lambda_{2}y\pp y\\
X_{0}\left(x,y\right) & := & u^{k}x\pp x+\left(1+\mu u^{k}\right)W_{0}\left(x,y\right)\,.
\end{eqnarray*}

\begin{thm}
\label{thm:res-sad-final} Assume that $\lie{Z,Y}$ is a non\textendash{}degenerate
Lie algebra of ratio $\delta$ with $\lambda\in\ww Q_{\leq0}$ and
such that $Z$ is not formally linearizable. Then there exists a unique
$\left(n,m\right)\in\ww Z^{2}$ such that $\left(Z,Y\right)$ is analytically
conjugate to $\left(\tilde{Z},\tilde{Y}\right)$ given by:
\begin{eqnarray*}
\tilde{Y} & = & x^{n}y^{m}\left(Q\circ u\right)^{\gamma}\left(c\left(Q\circ u\right)X_{0}+dW_{0}\right)\\
\tilde{Z} & = & \left(Q\circ u\right)X_{0}+\left(P\circ u\right)x^{n}y^{m}W_{0}\\
\gamma & = & -\frac{\delta\mu+\lambda_{2}n}{qk\lambda_{2}}\\
\delta & = & n\lambda_{1}+m\lambda_{2}
\end{eqnarray*}
where $P$ is a polynomial of degree at most $k-1$ and $Q$ at most
$k$, with $Q\left(0,0\right)=1$, and $\left(c,d\right)\in\ww C^{2}$.
The polynomial $P$ vanishes if $c\neq0$ or $\delta=0$, but is otherwise
unspecified. Moreover
\begin{itemize}
\item if $\delta\neq0$ then $Q\left(u\right)=1-\left(\lambda_{2}\frac{n}{\delta}+\mu\right)u^{k}$
and $\left(n,m\right)\notin\left(q,p\right)\ww Z$,
\item if $\delta=0$ then $Q$ is unspecified and $n=m=0$.
\end{itemize}
These expressions are unique up to diagonal changes of coordinates
$\left(x,y\right)\mapsto\left(\alpha x,\beta y\right)$ with $u\left(\alpha,\beta\right)^{k}=1$.
\end{thm}
Note that $\tilde{Z}$ is polynomial and $\tilde{Y}$ is a global
object with either $0$ or $k$ points of ramification. One could
adapt the upcoming proof to show that the result remains true with
$\tilde{Z}=\left(Q\circ u\right)\left(X_{0}+\left(P\circ u\right)u^{a}x^{n}y^{m}W_{0}\right)$
(for a different $P$, of course); the later expression might be more
useful for practical purposes.

\subsection{Preparation and general facts}

~

We need the following well known ingredients:
\begin{itemize}
\item According to a result by \noun{H.~Dulac}~\cite{Dulac} there exists
local analytic coordinates in which $Z$ writes
\begin{eqnarray*}
Z & = & UX\\
X & = & X_{0}+RY_{0}\\
Y_{0} & := & y^{\varepsilon}W_{0}
\end{eqnarray*}
where $U\left(0,0\right)=1$, $R$ is divisible by $u^{k+1}$ and
$\varepsilon\in\left\{ -1,1\right\} $. We define $\varepsilon:=1$
if $Z$ admits two transverse separatrices while if $Z$ has only
one separatrix then necessarily $\lambda=0$ and we set $\varepsilon:=-1$.
Notice that $\left[X_{0},Y_{0}\right]=\delta_{0}Y_{0}$ with 
\begin{eqnarray*}
\delta_{0} & := & \varepsilon\lambda_{2}\left(1+\mu u^{k}\right)\in\ker Y_{0}\cdot\,.
\end{eqnarray*}

\item Any meromorphic function $h$ such that $Z\ddt h=0$ is constant.
Any meromorphic function $b\neq0$ such that $Z\ddt b=Kb$ for some
holomorphic $K$ is of the form 
\begin{eqnarray*}
b & = & x^{N}y^{M}\beta\left(x,y\right)\,\,\,,\,\left(N,M\right)\in\ww Z^{2}\,\textrm{and }\beta\left(0,0\right)\neq0
\end{eqnarray*}
because $\left\{ x=0\right\} $ and $\left\{ y=0\right\} $ are the
only candidate separatrices of $Z$.
\item ~

\begin{lem}
\label{lem:res-sad-eq-h-form}Let $G\in\frml{x,y}$ and $\varepsilon\in\left\{ -1,0,1\right\} $
be given. There exists $F\in\frml{x,y}$ such that 
\begin{eqnarray*}
X\ddt F+\varepsilon\lambda_{2}\left(1+\mu u^{k}\right)F & = & G
\end{eqnarray*}
 if, and only if, $G$ does not contains terms $y^{-\varepsilon}u^{n}$
for $0\leq n\leq k$. If $\varepsilon\leq0$ the term $c:=\frac{\partial^{-\varepsilon}F}{\partial y^{-\varepsilon}}\left(0,0\right)$
is free and $F-c$ is unique. \end{lem}
\begin{proof}
Write $F\left(x,y\right):=\sum f_{a,b}x^{a}y^{b}$ and $G\left(x,y\right):=\sum g_{a,b}x^{a}y^{b}$.
The equation writes, for the monomial $x^{a}y^{b}$,
\begin{eqnarray*}
\left(\lambda_{1}a+\lambda_{2}\left(b+\varepsilon\right)\right)\left(f_{a,b}+\mu f_{a-kq,b-kp}\right)+\left(a-kq\right)f_{a-kq,b-kp}+o\left(a,b\right) & = & g_{a,b}
\end{eqnarray*}
where $o\left(a,b\right)$ stands for terms involving $f_{c,d}$ only
for $c+d<a+b-kq-kp$. We proceed by induction on the homogeneous degree
$a+b$  : while $\lambda_{1}a+\lambda_{2}\left(b+\varepsilon\right)\neq0$
the term $f_{a,b}$ is defined uniquely. On the contrary if $\left(a,b+\varepsilon\right)=l\left(q,p\right)$
with $l\in\ww Z$  :
\begin{enumerate}
\item if $0\leq l<k$ then $f_{a-kq,b-kp}=0$ so necessarily $g_{a,b}=0$. 
\item if $l=k$ then $g_{kq,kp}=0$ and $f_{0,-\varepsilon}$ is free if
$\varepsilon\leq0$ or is zero if $\varepsilon=1$.
\item if $l>k$ the term $f_{\left(l-k\right)p,\left(l-k\right)q}$ is well
determined.
\end{enumerate}
\end{proof}
\item Hence there exists a polynomial $Q\in\pol ku$ (of degree at most
$k$) satisfying $Q\left(0\right)=1$ such that $Z$ is formally tangentially
conjugate to $\left(Q\circ u\right)X$ by $\mathcal{T}:=\flw{QX}T{}$,
according to Proposition~\ref{pro:chg_coord}. Namely $Q$ is the
natural projection of $U$ on $\pol ku$ and
\begin{eqnarray*}
X\ddt T & = & \frac{1}{Q\circ u}-\frac{1}{U}\,.
\end{eqnarray*}

\item If $Z$ admits two separatrices, because $\left[X_{0},Y_{0}\right]=\delta_{0}Y_{0}$
and $Y_{0}\ddt u=0$ we deduce that $\left(Q\circ u\right)X_{0}$
is formally transversely conjugate to $\left(Q\circ u\right)X$ by
\begin{eqnarray}
\mathcal{N}\left(x,y\right): & = & \flw{Y_{0}}{N\left(x,y\right)}{\left(x,y\right)}\nonumber \\
 & = & \left(x\left(1-\varepsilon\lambda_{2}y^{\varepsilon}N\left(x,y\right)\right)^{-\varepsilon\lambda},y\left(1-\varepsilon\lambda_{2}y^{\varepsilon}N\left(x,y\right)\right)^{-\varepsilon}\right)\label{eq:res-sad-flot-Y}
\end{eqnarray}
where
\begin{eqnarray}
X\ddt N+\delta_{0}N & = & -R\,.\label{eq:res-sad-trans-eq-h}
\end{eqnarray}
 
\item The triple $\left(Q,k,\mu\right)$ is the formal invariant of conjugacy~\cite{Bruno}
for $Z$, up to the equivalence
\begin{eqnarray*}
\left(Q,k,\mu\right)\sim\left(\tilde{Q},\tilde{k},\tilde{\mu}\right) & \Leftrightarrow & \left(k,\mu\right)=\left(\tilde{k},\tilde{\mu}\right)\textrm{ and }Q\left(u\right)=\tilde{Q}\left(\alpha u\right)\textrm{ with }\alpha^{k}=1\,.
\end{eqnarray*}
In the following we fix $\left(k,\mu\right)$, which is the formal
invariant for the foliation underlying $Z$. We define
\begin{eqnarray*}
Z_{0} & := & \left(Q\circ u\right)X_{0}\,.
\end{eqnarray*}

\end{itemize}

\subsection{\label{sub:res-sad-form-aff-lie}Formal normal forms}

~

Here we study the Lie algebras $\mathcal{L}\left(Z_{0},\hat{Y}\right)$. 
\begin{prop}
\label{pro:res-sad-form-aff-lie}Assume that $\hat{Y}\neq0$ is a
formal meromorphic vector field such that $\left[Z_{0},\hat{Y}\right]=\delta\hat{Y}$
and $\hat{Y}\trs Z_{0}$. 
\begin{enumerate}
\item If $\delta\neq0$ there exists a unique $\left(n,m\right)\in\ww Z^{2}\backslash\left(q,p\right)\ww Z$
such that
\begin{eqnarray*}
\delta & = & n\lambda_{1}+m\lambda_{2}\\
Q\left(u\right) & = & 1-\left(\frac{n}{\delta}+\mu\right)u^{k}\\
\hat{Y} & = & x^{n}y^{m}\left(Q\circ u\right)^{\gamma}\left(cZ_{0}+dW_{0}\right)
\end{eqnarray*}
with $\left(c,d\right)\in\ww C\times\ww C^{*}$ and $\gamma=-\left(\delta\mu+n\right)/kq$.
\item If $\delta=0$ there exists $\left(c,d\right)\in\ww C\times\ww C^{*}$
such that
\begin{eqnarray*}
\hat{Y} & = & cZ_{0}+dW_{0}\,.
\end{eqnarray*}

\end{enumerate}
\end{prop}
\begin{proof}
Observe that $W_{0}$ is holomorphic and commutes with $Z_{0}$. Write
$\hat{Y}=aZ_{0}+bW_{0}$ with $a$ and $b$ in $\fmero{x,y}$. Then
\begin{eqnarray*}
Z_{0}\ddt a & = & \delta a\\
Z_{0}\ddt b & = & \delta b\,.
\end{eqnarray*}
Since $b\neq0$ we can write $b\left(x,y\right)=x^{n}y^{m}\beta\left(x,y\right)$
for $\left(n,m\right)\in\ww Z^{2}$ and $\beta\left(0,0\right)\neq0$.
We find  :
\begin{eqnarray*}
\left(Q\circ u\right)\left(\left(n\lambda_{1}+m\lambda_{2}\right)\left(1+\mu u^{k}\right)+nu^{k}\right)\beta+Z_{0}\ddt\beta & = & \delta\beta\,.
\end{eqnarray*}
Define $B:=\log\beta\in\frml{x,y}$. Lemma~\ref{lem:res-sad-eq-h-form}
implies that
\begin{eqnarray}
Z_{0}\ddt B & = & \delta-(Q\circ u)\left(\left(n\lambda_{1}+m\lambda_{2}\right)\left(1+\mu u^{k}\right)+nu^{k}\right)\label{eq:res-sad-eq-h-1}
\end{eqnarray}
cannot contain terms of the form $u^{l}$ for $0\leq l\leq k$. We
immediately get that $\delta=n\lambda_{1}+m\lambda_{2}$. 

Write $Q\left(u\right)=1+\sum_{j=1}^{k}v_{j}u^{k}$ and assume first
that $\delta\neq0$. We immediately derive that $v_{j}=0$ for $0<j<k$,
while
\begin{eqnarray*}
n & = & -\delta\left(v_{k}+\mu\right)\,.
\end{eqnarray*}
Thus 
\begin{eqnarray*}
Q\circ u & = & 1-\left(\frac{n}{\delta}+\mu\right)u^{k}\,,
\end{eqnarray*}
which determines $n$ completely (and thus $m$ since $\delta=n\lambda_{1}+m\lambda_{2}$).
On the one hand for all $\gamma\in\ww C$ 
\begin{eqnarray*}
Z_{0}\ddt\gamma\log Q\circ u & = & \gamma qu^{k+1}Q'\circ u\\
 & = & \gamma qk\lambda_{2}v_{k}u^{2k}
\end{eqnarray*}
whereas on the other hand
\begin{eqnarray*}
Z_{0}\ddt B & = & -\delta v_{k}^{2}u^{2k}\,.
\end{eqnarray*}
The uniqueness condition in Lemma \ref{lem:res-sad-eq-h-form} allows
us to find $C\in\ww C$ such that
\begin{eqnarray*}
B & = & \gamma\log Q\circ u+C\,
\end{eqnarray*}
with 
\begin{eqnarray*}
\gamma & := & -\frac{\delta\mu+n}{qk}\,.
\end{eqnarray*}
To conclude
\begin{eqnarray*}
b & = & dx^{n}y^{m}\left(Q\circ u\right)^{\gamma}\\
a & = & cx^{n}y^{m}\left(Q\circ u\right)^{\gamma}\,.
\end{eqnarray*}

Finally when $\delta:=0$ the above computations show that $n=0$,
then $m=0$ and $\beta$ must be constant.
\end{proof}

\subsection{\label{sub:res-sed-tang-rigid}Tangential rigidity}

~

Here we do not assume anything on the number of separatrices of $Z$
or the existence of a non\textendash{}degenerate Lie algebra $\mathcal{L}\left(Z,Y\right)$.
We end the proof of Theorem~\ref{thm:rigid-tang}:
\begin{prop}
\label{pro:res-sad-tang-rigid}There exists a closed time\textendash{}form
$\tau$ for $Z$ if, and only if, $Z$ is analytically conjugate to
$Q\left(u\right)X$ by a tangential change of coordinates.
\end{prop}
We proceed along the same steps as taken in Section~\ref{sub:irrat-tang-rigid}.
\begin{enumerate}
\item If the polar locus of $\tau_{0}:=\tau-\frac{1}{Q\circ u}\tau_{X}$
is a separatrix of $Z$, where $\tau_{X}:=\frac{du}{u^{k+1}}$, we
have 
\begin{eqnarray*}
\tau_{0} & = & \dd{\left(\frac{A\left(x,y\right)}{x^{N}y^{M}}+\alpha\log x+\beta\log y\right)}\\
 & = & \dd F
\end{eqnarray*}
with $A$ holomorphic and $\left(N,M,\alpha,\beta\right)\in\ww Z_{\geq0}^{2}\times\ww C^{2}$.
Besides
\begin{eqnarray*}
X\ddt F & = & \frac{1}{Q\circ u}-\frac{1}{U}\\
 & = & X\ddt T\,.
\end{eqnarray*}

\item We need to show that $A=x^{N}y^{M}\tilde{A}$ and $\alpha=\beta=0$.
Once this is done the uniqueness condition in Lemma~\ref{lem:res-sad-eq-h-form}
implies that $T=\tilde{A}-\tilde{A}\left(0,0\right)$ is a convergent
power series.
\item Any closed time\textendash{}form $\tau$ of $Z$ is of the form 
\begin{eqnarray*}
\tau & = & \frac{1}{U}\tau_{X}+\frac{f}{\frak{p}}\omega
\end{eqnarray*}
for some holomorphic $1$\textendash{}form $\omega$ with $\omega\left(Z\right)=0$.
The polar locus $\left\{ \frak{p}=0\right\} $ is again a separatrix
of $Z$.
\end{enumerate}
We state now in detail the second step but we skip the proof since
it is very similar to the one of Lemma~\ref{lem:irrat-tang-rigid}.
\begin{lem}
\label{lem:res-sad-tang-rigid}Let $\mathcal{A}$ be the linear space
over $\ww C$ of all formal expressions 
\[
\frac{A\left(x,y\right)}{x^{N}y^{M}}+\alpha\log x+\beta\log y
\]
with $A\in\frml{x,y}$ and $\left(\alpha,\beta,N,M\right)\in\ww C^{2}\times\ww Z_{\geq0}^{2}$.
Assume that $F\in\mathcal{A}$ is such that $X\ddt F=G\in\frml{x,y}$.
Then there exists $\hat{F}\in\frml{x,y}$ and a unique $\alpha\in\ww C$
such that
\begin{eqnarray*}
F\left(x,y\right) & = & \hat{F}\left(x,y\right)-G\left(0,0\right)\frac{1}{qku^{k}}+\alpha\left(\left(1-\mu p\right)\log y-\mu q\log x+\frac{1}{ku^{k}}\right)\,.
\end{eqnarray*}
The power series $\hat{F}-\hat{F}\left(0,0\right)$ is unique. On
the other hand every $G\in\frml{x,y}$ writes $X\ddt F$ for some
$F\in\mathcal{A}$ with $\alpha=0$.
\end{lem}

\subsection{\label{sub:res-sed-trans-rigid}Transversal rigidity}

~

Let us write
\begin{eqnarray*}
\hat{Y} & = & L\left(cZ_{0}+dy^{-\varepsilon}Y_{0}\right)\\
L & := & x^{n}y^{m}\left(Q\circ u\right)^{\gamma}
\end{eqnarray*}
with $c\in\ww C$, $d\in\ww C^{*}$ as in Proposition~\ref{pro:res-sad-form-aff-lie}.
Since $Y=\mathcal{N}^{*}\hat{Y}$ is meromorphic we deduce that
\begin{eqnarray*}
cL\circ\mathcal{N} & \textrm{and} & \left(y^{-\varepsilon}L\right)\circ\mathcal{N}\left(1+Y_{0}\ddt N\right)^{-1}
\end{eqnarray*}
 are meromorphic. Taking~(\ref{eq:res-sad-flot-Y}) into account
we derive that
\begin{eqnarray*}
L\circ\mathcal{N} & = & L\left(1-\varepsilon\lambda_{2}y^{\varepsilon}N\right)^{-\varepsilon\left(m+n\lambda\right)}\\
\left(y^{-\varepsilon}L\right)\circ\mathcal{N} & = & y^{-\varepsilon}L\left(1-\varepsilon\lambda_{2}y^{\varepsilon}N\right)^{1-\varepsilon\left(m+n\lambda\right)}
\end{eqnarray*}
which means, if $c\neq0$, that we are done: $N$ is convergent since
$\varepsilon\left(m+n\lambda\right)=\varepsilon\delta/\lambda_{2}\neq0$.
On the other hand, if $c=0$, then
\begin{eqnarray}
\left(1+Y_{0}\ddt N\right)\left(1-\varepsilon\lambda_{2}y^{\varepsilon}N\right)^{\varepsilon\delta/\lambda_{2}-1} & = & A\label{eq:res-sad-trans-rigid-eq-h}
\end{eqnarray}
with $A$ meromorphic. Indeed, either $\varepsilon=1$ and the claim
is clear, or $\varepsilon=-1$ and $\varepsilon\delta/\lambda_{2}=-m$
so that $\left(1+\lambda_{2}y^{-1}N\right)^{-\delta/\lambda_{2}-1}=y^{m+1}\left(y+\lambda_{2}N\right)^{-m-1}$. 
\begin{lem}
There exists a convergent power series $\tilde{N}$ satisfying equation~(\ref{eq:res-sad-trans-rigid-eq-h})
and such that $\tilde{N}\left(0,0\right)=Y_{0}\ddt\tilde{N}\left(0,0\right)=0$.\end{lem}
\begin{proof}
Noticing that $Y_{0}\ddt y^{-\varepsilon}=-\varepsilon\lambda_{2}$
we deduce that 
\begin{eqnarray*}
Y_{0}\ddt\left(y^{-\varepsilon}-\varepsilon\lambda_{2}N\right)^{\varepsilon\delta/\lambda_{2}} & = & -\delta\left(y^{-\varepsilon}-\varepsilon\lambda_{2}\right)^{\varepsilon\delta/\lambda_{2}-1}\left(1+Y_{0}\ddt N\right)\\
 & = & -\delta y^{-\delta/\lambda_{2}+\varepsilon}A\,.
\end{eqnarray*}
Assume first that $\varepsilon=1$ and write $\left(1-\lambda_{2}yN\left(x,y\right)\right)^{\delta/\lambda_{2}}=\sum N_{a,b}x^{a}y^{b}$.
The previous equation writes 
\begin{eqnarray*}
\sum_{a+b>0}\left(\lambda_{1}a+\lambda_{2}b-\delta\right)N_{a,b}x^{a}y^{b} & = & -\delta A\left(x,y\right)
\end{eqnarray*}
and we know that a formal solution exists. As a consequence $A\left(x,y\right)=\sum_{a,b\geq0}A_{a,b}x^{a}y^{b}$
is holomorphic and, since $\delta=\lambda_{1}n+\lambda_{2}m$,
\begin{eqnarray*}
\frac{\lambda_{2}}{q}\left(-p\left(a-n\right)+q\left(b-m\right)\right)N_{a,b} & = & A_{a,b}\,.
\end{eqnarray*}
When $\left(a,b\right)\notin\mathcal{A}:=\left(q,p\right)\ww Z+\left(n,m\right)$
the estimate
\begin{eqnarray*}
\lambda_{2}\left|N_{a,b}\right| & \leq & q\left|A_{a,b}\right|
\end{eqnarray*}
 proves the convergence of 
\begin{eqnarray*}
N_{0}\left(x,y\right) & := & \sum_{\left(a,b\right)\in\ww Z_{\geq0}^{2}\backslash\mathcal{A}}N_{a,b}x^{a}y^{b}\,.
\end{eqnarray*}
We finally set 
\begin{eqnarray*}
\tilde{N}\left(x,y\right) & := & \frac{1}{\lambda_{2}y}\left(1-N_{0}^{\lambda_{2}/\delta}\right)
\end{eqnarray*}
 which is of course a formal power series.

Assume now that $\varepsilon=-1$ so that $Y_{0}=\lambda_{2}\pp y$.
Let $\left\{ B=0\right\} $ be the (formal) vanishing locus of $y+\lambda_{2}N$.
The power series $B$ is not zero since $X_{0}$ admits $\left\{ y=0\right\} $
for separatrix. Because $N\left(0,0\right)=\left(Y_{0}\ddt N\right)\left(0,0\right)=0$
we deduce that $B\left(x,y\right)=y-\hat{s}\left(x\right)$ for some
$\hat{s}\in\frml x$ so that there exists $\hat{N}\in\frml{x,y}$
such that
\begin{eqnarray*}
y+\lambda_{2}N & = & \left(y-\hat{s}\right)\exp\hat{N}\,.
\end{eqnarray*}
We obtain that
\begin{eqnarray*}
y^{-m-1}A & = & Y_{0}\ddt\left(y+\lambda_{2}N\right)^{-m}\\
 & = & \left(y-\hat{s}\right)^{-m-1}\exp\left(-m\hat{N}\right)\left(1+\left(y-\hat{s}\right)\pp y\hat{N}\right)
\end{eqnarray*}
so that the hypothetical poles of $y^{-m-1}A$ are of the form $\left\{ B=0\right\} $
and of order $m+1\neq1$. Hence $y^{-m-1}A$ admits a meromorphic
primitive $E$ with respect to $\pp y$ such that $E\left(x,0\right)=0$
and 
\begin{eqnarray*}
E & = & \left(y-\hat{s}\right)^{-m}F
\end{eqnarray*}
with $F\in\frml{x,y}$, $F\left(0,0\right)=1$. Hence $\left(y-\hat{s}\right)F^{-1/m}$
is holomorphic and the following holomorphic function solves our problem
 : 
\begin{eqnarray*}
\tilde{N}\left(x,y\right) & := & \frac{E^{-1/m}-y}{\lambda_{2}}\,.
\end{eqnarray*}
\end{proof}
\begin{cor}
\label{cor:res-sad-trans-rigid}The couple $\left(Z,Y\right)$ is
analytically conjugate to some $\left(Z_{0}+K\hat{Y},\hat{Y}\right)$
where $\hat{Y}\cdot K=0$. If moreover $Y\trs Y_{0}$ then $K=0$.\end{cor}
\begin{proof}
On the one hand we have already discussed the case $Y\trs Y_{0}$.
On the other hand we just proved that $\tilde{\mathcal{N}}:=\flw{Y_{0}}{\tilde{N}}{}$
is an analytic conjugacy between $\hat{Y}$ and $Y$. Set $\tilde{Z}:=\tilde{\mathcal{N}}_{*}Z=Z_{0}+K\hat{Y}$
with $K$ meromorphic. Since $\left[\tilde{Z},\hat{Y}\right]=\delta\hat{Y}$
it turns out that $\hat{Y}\ddt K=0$. 
\end{proof}

\subsection{Final reduction~: proof of Theorem~\ref{thm:res-sad-final}}

~

We only give a sketch of proof: computations for saddle\textendash{}nodes
have already been done in~\cite{Tey-ExSN}, and they generalize straightforwardly
to resonant saddles. Assume that $X=X_{0}+\left(K\circ u\right)x^{n}y^{m}W_{0}$
and that $\tilde{Y}=dx^{n}y^{m}\left(Q\circ u\right)^{\gamma}W_{0}$
(we already know that $K=0$ if $\tilde{Y}\trs W_{0}$). By playing
with linear transformations $\left(x,y\right)\mapsto\left(x,\alpha y\right)$,
one can choose $d=1$. We seek a function $F\in\germ u$ such that
\begin{eqnarray*}
\mathcal{N}_{F} & := & \flw{\tilde{Y}}{F\circ u}{}
\end{eqnarray*}
satisfies
\begin{eqnarray*}
\mathcal{N}_{F}^{*}X & = & X_{0}+\left(P\circ u\right)x^{n}y^{m}W_{0}\,.
\end{eqnarray*}
Because $W_{0}\ddt\left(Q\circ u\right)=0$ we can write $\mathcal{N}_{F}=\flw{Y_{0}}{FQ\circ u}{}$
where $Y_{0}=x^{n}y^{m}W_{0}$, so that $\mathcal{N}_{F}^{*}\left(\left(K\circ u\right)Y_{0}\right)=\left(K\circ u\right)Y_{0}$.
We end up with the cohomological equation
\begin{eqnarray*}
X_{0}\ddt\left(FQ\circ u\right)+\delta_{0}FQ\circ u & = & \left(K-P\right)\circ u
\end{eqnarray*}
where $\left[X_{0},Y_{0}\right]=\delta_{0}Y_{0}$, which can be rewritten
\begin{eqnarray*}
X_{0}\ddt\left(x^{n}y^{m}FQ\circ u\right) & = & -\delta_{0}x^{n}y^{m}\left(K-P\right)\circ u\,.
\end{eqnarray*}
In the above\textendash{}mentioned reference the reader will find
a way of (explicitly) constructing a polynomial $P$, satisfying the
sought properties, such that the equation admits an analytic solution
$F$. 
\begin{rem}
In fact we prove in~\cite{Tey-ExSN} that $\tilde{Z}$ is analytically
conjugate to $Z_{0}$ if, and only if, $P=0$. Hence $\lie{Z_{0},\tilde{Y}}$
and $\lie{Z_{0}+\left(P\circ u\right)\tilde{Y},\tilde{Y}}$ are formally
conjugate affine Lie algebras which are not analytically conjugate.
This remark provides a family of examples complementing those of~\cite{Cervo}.
\end{rem}
\bigskip{}
In the result of rigidity we presented here the main obstruction regarding
the convergence of the formal objects we built is the collinearity
of the vector fields $\tilde{Y}$ and $Y_{0}$. In other terms denote
by $\Omega:=\left(Z_{0},Y_{0},\hat{Y}\right)$ the (singular) $3$\textendash{}web
corresponding to the formal normal form of $\lie{Z,Y}$: if $\Omega$
is non\textendash{}degenerate then $\lie{Z,Y}$ is analytically conjugate
to $\lie{Z_{0},\hat{Y}}$ whereas in the opposite case such a result
is not true in general. The geometrical meaning of that fact is that
we have constructed the formal conjugacies using the flows along $Z$
and $Y$, while the convergence of the power series involved related
to those vector fields being transverse to each\textendash{}others.
When this is no longer the case we cannot infer the convergence of
the whole power series, but only of a sub\textendash{}series which
«does not contain» first\textendash{}integrals of $Y_{0}$ since the
only information we have is of the type «$Y_{0}\ddt N$ is convergent».

\section{\label{sec:delta-lattice}Classification by the delta\textendash{}lattice }
\begin{defn}
\label{def_delta-lattice} Let $Z$ be a germ of a meromorphic vector
field at $\left(0,0\right)$. We define the \textbf{transverse structures}
of $Z$ as the set 
\begin{eqnarray*}
\tx{TS}\left(Z\right) & := & \left\{ \left(\delta,Y\right)\,:\, Y\trs Z,\,\left[Z,Y\right]=\delta Y\right\} 
\end{eqnarray*}
and call $\Delta\left(Z\right)$ its \textbf{delta\textendash{}lattice}:
\begin{eqnarray*}
\Delta\left(Z\right) & := & \left\{ \delta\,:\,\left(\exists Y\right)\,\left(\delta,Y\right)\in\tx{TS}\left(Z\right)\right\} \,.
\end{eqnarray*}

\end{defn}
In this section we focus once more on holomorphic germs of a vector
field. As before if $Z\left(0,0\right)=0$ we define $\left\{ \lambda_{1},\lambda_{2}\right\} $
the spectrum of the linear part of $Z$ at $\left(0,0\right)$. When
$Z$ is not nilpotent we label the eigenvalues in such a way that
$\lambda_{2}\neq0$ and set $\lambda:=\nf{\lambda_{1}}{\lambda_{2}}$.
The delta\textendash{}lattice of a holomorphic vector fields characterize
its dynamical class.
\begin{prop}
\label{pro:delta-classif}Assume that $\Delta\left(Z\right)\neq\emptyset$.
Then $\Delta\left(Z\right)$ is an affine $\ww Z$\textendash{}lattice
corresponding to exactly one of the following cases. 
\begin{enumerate}
\item $\Delta\left(Z\right)=\ww C$ if, and only if, $Z$ is regular at
$\left(0,0\right)$.
\item $\Delta\left(Z\right)=\lambda_{1}\ww Z\oplus\lambda_{2}\ww Z$ of
rank $2$ if, and only if, $Z$ is not nilpotent and $\lambda\notin\ww R$.
\item $\Delta\left(Z\right)=\lambda_{1}\ww Z\oplus\lambda_{2}\ww Z$ dense
in a real line if, and only if, $Z$ is not nilpotent, analytically
linearizable and $\lambda\in\ww R\backslash\ww Q$.
\item $\Delta\left(Z\right)=\delta\ww Z$ of rank 1 ($\delta\neq0$) if,
and only if, $Z$ is not nilpotent, $\lambda\in\ww Q$ and $Z$ is
analytically conjugate to its formal normal form.
\item $\Delta\left(Z\right)=\left\{ \delta\right\} $ if, and only if, $Z$
is a resonant saddle or a saddle\textendash{}node not analytically
conjugate to its formal normal form, or has nilpotent linear part.
In the latter case $\delta=0$.
\end{enumerate}
\end{prop}
\begin{rem}
In~(4) the formal normal form considered is either the linear part
of $Z$ (non\textendash{}resonant singularity) or its Dulac\textendash{}Poincaré
normal form. The latter case encompasses resonant saddles ($\lambda\in\ww Q_{<0}$)
and saddle\textendash{}nodes ($\lambda=0$) studied in Theorem~\ref{thm:res-sad-final}
as well as resonant nodes ($\lambda\in\ww N\cup\nf 1{\ww N}$ and
the normal form is given \emph{e.g.} when $\lambda\in\ww N$ by $\lambda_{1}\left(x+y^{\lambda}\right)\pp x+\lambda_{2}y\pp y$).\end{rem}
\begin{proof}
Assume there exists $\left(\delta,Y\right),\,\left(\tilde{\delta},\tilde{Y}\right)\in\tx{TS}\left(Z\right)$
and write $\tilde{Y}=aZ+bY$ with $a,\, b\in\mero{x,y}$ and $b\neq0$.
Then $Z\cdot a=\tilde{\delta}a$ and $Z\cdot b=\left(\tilde{\delta}-\delta\right)b$.
For $n\in\ww Z$ the meromorphic germ $b^{n}$ solves $Z\cdot b^{n}=n\left(\tilde{\delta}-\delta\right)b^{n}$
so that $\left(\delta+n\left(\tilde{\delta}-\delta\right),b^{n}Y\right)\in\tx{TS}\left(Z\right)$
and $\Delta\left(Z\right)$ is an affine $\ww Z$\textendash{}lattice.

The remaining of the proof is only a restatement of both classical
results and results established in the previous sections. We indicate
shortly below why the condition on $Z$ is sufficient, the necessity
of said condition then follows from the exhaustion of all dynamical
types of germs of a vector field by the above discrimination. In case~(1)
if $Z\left(0,0\right)\neq0$ the rectification theorem states that,
in some appropriate local chart, $Z=\pp y$. Hence being given any
$\delta\in\ww C$ the vector field $Y:=e^{\delta y}\pp x$ induces
a non\textendash{}degenerate Lie algebra $\lie{Z,Y}$ of ratio $\delta$.
Case~(2) is a straightforward consequence of Poincaré's linearization
theorem and of direct computations analogous to those we carried out
for the last part of the proof of Corollary~\ref{cor:non-isolated}.
In fact all linearizable cases are dealt with in the same way, which
accounts for~(3) \emph{via }the use of Theorem~\ref{thm:irrat-th},
as well as~(4) for non\textendash{}resonant singularities. Cases~(4)
and~(5) for resonant singularities are taken care of by Theorem~\ref{thm:res-sad-final}.
Notice that if $Z$ is nilpotent then Lemma~\ref{lem:degen} ensures
that $\Delta\left(Z\right)=\left\{ 0\right\} $ (here the holomorphy
of $Z$ is an essential part of the argument).
\end{proof}

\section{\label{sec:galois}Computation of Galois\textendash{}Malgrange's
groupoid}

We assume in this section that $Z$ is a germ of a \emph{meromorphic
}vector field at $\left(0,0\right)$.

\subsection{Definitions}

~

We do not introduce in full details what the Galois\textendash{}Malgrange
groupoid for a meromorphic vector field $Z$ (or foliation $\fol Z$)
is. We refer to~\cite{Malg} and~\cite{Casa} for precise definitions
and general properties. Roughly speaking a $\mathcal{D}$\textendash{}groupoid
is a groupoid (pseudo\textendash{}group) of germs of a diffeomorphism
\begin{eqnarray*}
\Gamma\,:\,\left(\ww C^{2},p\right) & \to & \left(\ww C^{2},q\right)
\end{eqnarray*}
solutions to some differentially stable (sheaf of) ideal $\mathcal{I}$
of partial differential equations, meromorphic in $\Gamma$ and polynomial
in its derivatives, and where $p,\, q$ are taken in a polydisc $\Delta$.
The important point here is that $\mathcal{I}$ may only be defined
outside an analytic hypersurface, which in our case will be located
along the polar locus of $Z$ and $Y$, as well as along their tangency
locus. To such a $\mathcal{D}$\textendash{}groupoid is associated
a $\mathcal{D}$\textendash{}algebroid, which is a (sheaf of) Lie
algebra whose elements are solutions to the linearized equations near
the identity element. Not every $\mathcal{D}$\textendash{}algebroid
is integrable (\emph{i.e. }that of a $\mathcal{D}$\textendash{}groupoid). 
\begin{defn}
Let $Z$ be a germ of a holomorphic vector field in $\ww C^{2}$ with
underlying foliation $\fol Z$.
\begin{enumerate}
\item The \textbf{Galois\textendash{}Malgrange groupoid} of $\fol Z$ is
the smallest $\mathcal{D}$\textendash{}groupoid $\gal{\fol Z}$ whose
$\mathcal{D}$\textendash{}algebroid $\dlie{\fol Z}$ contains the
sheaf $\mathcal{O}Z$ of vector fields tangent to $\fol Z$. 
\item The \textbf{Galois\textendash{}Malgrange groupoid} of the meromorphic
vector field $Z$ is the smallest $\mathcal{D}$\textendash{}groupoid
$\gal Z$ whose $\mathcal{D}$\textendash{}algebroid $\dlie Z$ contains
$Z$.
\end{enumerate}
\end{defn}
\begin{rem}
~
\begin{enumerate}
\item Obviously $\gal Z<\gal{\fol Z}$ as $\dlie Z<\dlie{\fol Z}$.
\item The sheaf of germs of an isotropy of $\fol Z$ (\emph{resp.} $Z$)
form a $\mathcal{D}$\textendash{}groupoid $\aut{\fol Z}$ (\emph{resp.}
$\aut Z$) containing $\gal{\fol Z}$ (\emph{resp.} $\gal Z$). 
\end{enumerate}
\end{rem}
Knowing whether $\gal{\fol Z}$ is a proper subgroupoid of $\aut{\fol Z}$
pertains to the problem of Liouvillian integrability of the underlying
differential equation, as we explain now. In the case of a foliation
or vector field equations of $\mathcal{I}$ defining elements $\Gamma\in\gal{\fol Z}$
can be restricted to the transverse direction, yielding an ideal of
equations $\mathcal{I}_{\trs}$ of a transverse $\mathcal{D}$\textendash{}groupoid
(see~\cite[Lemme 2.4]{Casa}). This transverse $\mathcal{D}$\textendash{}groupoid
encodes all the information regarding the question of Liouvillian
integrability of the foliation. The information is concentrated in
the \textbf{transverse rank} $\ell_{Z}$ of $\fol Z$, which is the
rank of $\mathcal{I}_{\trs}$. This rank does not depend on the transverse
disc considered nor on the local analytic charts in which $Z$ is
expressed.
\begin{namedthm}[Theorem of Casale\textendash{}Malgrange ]
 Let $\ell_{Z}$ be the transverse rank of $\fol Z$. The following
properties are equivalent.
\begin{enumerate}
\item $\ell_{Z}<\infty$.
\item $\ell_{Z}\leq3$.
\item $\gal{\fol Z}$ is a proper subgroupoid of $\aut{\fol Z}$.
\item $\fol Z$ admits a Godbillon\textendash{}Vey sequence of finite length
$\ell_{Z}$ and no sequence of lesser length.
\end{enumerate}
\end{namedthm}

\subsection{Integrability theorem}

~
\begin{defn}
Let $Z$ be a meromorphic vector field on a small polydisc $\Delta$
around $\left(0,0\right)$. We denote by $\sol$ the coherent sheaf
of germs of a holomorphic function $f$ at points of $\Delta\backslash\left\{ \left(0,0\right)\right\} $,
called $\delta$\textendash{}solutions, such that 
\begin{eqnarray*}
Z\cdot f & = & \delta f\,.
\end{eqnarray*}
In particular $\sol[][0]$ is the sheaf of holomorphic first\textendash{}integrals
of $Z$.\end{defn}
\begin{rem}
\label{rem_sol_space}A straightforward computation shows the inclusion,
for every $\alpha,\,\beta\in\ww C$,
\begin{eqnarray*}
\sol[Z][\alpha]\sol[Z][\beta] & \subset & \sol[Z][\alpha+\beta]\,.
\end{eqnarray*}
Besides if $f\in\sol[Z][\alpha]$ then $Y\cdot f\in\sol[Z][\alpha+\delta]$
for any vector field $Y$ such that $\left[Z,Y\right]=\delta Y$.
\end{rem}
We are now able to state our theorem:
\begin{thm}
\label{thm:compute_gal}Assume that the transverse structures $\tx{TS}\left(Z\right)$
is not empty and the delta\textendash{}lattice is not reduced to $\left\{ 0\right\} $.
For $\left(\delta,Y\right)\in\tx{TS}\left(Z\right)$ denote by $\aut{Z,\lie{Z,Y}}$
the pseudo\textendash{}group $\aut Z\cap\aut{\lie{Z,Y}}$ of invariance
of the pair $\left(Z,\lie{Z,Y}\right)$. Then the following properties
hold.
\begin{enumerate}
\item ~
\begin{eqnarray*}
\gal Z & = & \bigcap_{\left(\delta,Y\right)\in\tx{TS}\left(Z\right)}\aut{Z,\lie{Z,Y}}\,.
\end{eqnarray*}

\item There exists a sheaf of Lie algebras of dimension at most $2$, $\dtrans$,
such that 
\begin{eqnarray*}
\dlie Z & = & \ww CZ\oplus\dtrans\,.
\end{eqnarray*}
It is actually an integrable $\mathcal{D}$\textendash{}algebra whose
corresponding $\mathcal{D}$\textendash{}goupoid $\trans[Z]$ is isomorphic
to the transverse $\mathcal{D}$\textendash{}groupoid of $\fol Z$.
The dimension of $\dtrans$ is therefore $\ell_{Z}$.
\end{enumerate}
\end{thm}
In order to characterize elements $\Gamma\in\aut Z$ actually belonging
to $\gal Z$ we need to find a minimal set of proper equations satisfied
by $\Gamma$. These equations are obtained by studying the action
of $\aut Z$ on the other generators $Y$ of non\textendash{}degenerate
Lie algebras $\lie{Z,Y}$ of ratio $\delta\in\ww C$.
\begin{lem}
\label{lem:charac_aut}$\Gamma\in\aut{Z,\lie{Z,Y}}$ if, and only
if,
\begin{eqnarray*}
\begin{cases}
\Gamma^{*}Z & =Z\\
\Gamma^{*}Y & =d_{\Gamma}Z+c_{\Gamma}Y\,\,\,\,\,\,\,\,\,\,\,\,\mbox{for some }\left(c_{\Gamma},d_{\Gamma}\right)\in\ww C_{\neq0}\times\ww C
\end{cases}
\end{eqnarray*}
with $\delta d_{\Gamma}=0$.\end{lem}
\begin{proof}
We have $\left[Z,\Gamma^{*}Y\right]=\delta\Gamma^{*}Y$ which means
$\Gamma^{*}Y=d_{\Gamma}Z+c_{\Gamma}Y$ for unique $d_{\Gamma}\in\sol[][\delta]$
and $c_{\Gamma}\in\sol[][0]$. The condition $\lie{Z,\Gamma^{*}Y}=\lie{Z,Y}$
then implies $c_{\Gamma},\, d_{\Gamma}\in\ww C$ and finally $d_{\Gamma}=0$
if $\delta\neq0$.
\end{proof}

\subsubsection{Action of $\aut Z$ on $\lie{Z,Y}$}

~
\begin{lem}
\label{lem:aut_action_on_Y}Let $\lie{Z,Y}$ be a non\textendash{}degenerate
Lie algebra of ratio $\delta$, meromorphic on a small polydisc $\Delta$
centered at $\left(0,0\right)$. Take $\Gamma\in\aut Z$ defined on
a connected neighborhood $U'\subset\Delta$ of a point $p_{\Gamma}\in\Delta\backslash\left\{ \left(0,0\right)\right\} $
ranging in $U\supset U'$, such that:
\begin{itemize}
\item $Z|_{U}$ and $Y|_{U}$ are holomorphic and transverse,
\item $Z|_{U}$ is rectifiable.
\end{itemize}
Then the following properties hold.
\begin{enumerate}
\item There exists $T_{\Gamma}\in\sol[][0]$ and $N_{\Gamma}\in\sol[][-\delta]$
such that 
\begin{eqnarray*}
\Gamma & = & \flw Y{N_{\Gamma}}{}\circ\flw Z{T_{\Gamma}}{}\,.
\end{eqnarray*}

\item Moreover
\begin{eqnarray*}
\Gamma^{*}Y & = & \frac{\exp\left(\delta T_{\Gamma}\right)}{1+Y\cdot N_{\Gamma}}\left(Y-\left(Y\cdot T_{\Gamma}\right)Z\right)
\end{eqnarray*}
with $Y\cdot N_{\Gamma}\in\sol[][0]$ and $Y\cdot T_{\Gamma}\in\sol$.
\end{enumerate}
\end{lem}
\begin{proof}
~
\begin{enumerate}
\item Let $\psi\,:\, U\to\tilde{U}$ be a rectifying chart of $Z$ and denote
$\left(t,z\right)$ a system of coordinates on $\tilde{U}$, that
is $\psi^{*}Z=\tilde{Z}:=\pp t$ and $z$ is the «transverse» coordinates.
We ask that $\psi$ be holomorphic, one to one and send $p_{\Gamma}$
on $\left(0,0\right)$; set $\tilde{q}:=\psi\left(q_{\Gamma}\right)$.
Since $\tilde{\Gamma}:=\psi^{*}\Gamma$ belongs to $\aut{\pp t}$
we have
\begin{eqnarray*}
\tilde{\Gamma}\left(t,z\right) & = & \left(t+\alpha\left(z\right),\beta\left(z\right)\right)
\end{eqnarray*}
 for some $\alpha,\,\beta\in\mathcal{O}\left(\psi\left(U'\right)\right)$
with $\left(\alpha\left(0\right),\beta\left(0\right)\right)=\tilde{q}$.\\
Write $\tilde{Y}:=\psi^{*}Y$, so that $\left[\tilde{Z},\tilde{Y}\right]=\delta\tilde{Y}$.
For $T,\, N\in\mathcal{O}\left(\psi\left(U'\right)\right)$ apply
Proposition~\ref{pro:chg_coord} with $\mathcal{N}:=\flw{\tilde{Y}}N{}$
and $\mathcal{T}:=\flw{\tilde{Z}}T{}$. We seek $\eta\in\germ z$
such that $N\left(t,z\right):=\eta\left(z\right)\exp\left(-\delta t\right)\in\sol[\pp t][-\delta]$
solves
\begin{eqnarray*}
\pi_{z}\circ\mathcal{N}\left(t,z\right) & = & \beta\left(z\right)\,,
\end{eqnarray*}
 where $\pi_{z}$ is the natural projection $\left(t,z\right)\mapsto z$.
This function is given by the implicit function theorem applied to
the map
\begin{eqnarray*}
g\left(t,z,\tau\right) & := & \pi_{z}\circ\flw{\tilde{Y}}{\tau}{\left(t,z\right)}\,,
\end{eqnarray*}
 since $\frac{\partial g}{\partial\tau}\left(t,z,\tau\right)=\pi_{z}\circ\tilde{Y}\circ\flw{\tilde{Y}}{\tau}{\left(t,z\right)}$
does not vanish for $\tilde{Z}$ and $\tilde{Y}$ are transverse on
$\tilde{U}$. Now $\left(\mathcal{N}^{-1}\circ\tilde{\Gamma}\right)\in\aut{\pp t}$
is the identity in the transverse variable and therefore is of the
form $\left(t,z\right)\mapsto\flw{\tilde{Z}}{\tilde{\alpha}\left(z\right)}{\left(t,z\right)=\left(t+\tilde{\alpha}\left(z\right),z\right)}$
with $\tilde{\alpha}\in\mathcal{O}\left(\psi\left(U'\right)\right)$.
Back in the original coordinates we deduce the result with $N_{\Gamma}:=N\circ\psi^{-1}$
and $T_{\Gamma}:=T\circ\psi^{-1}$.
\item Write $\Gamma=\mathcal{N}\circ\mathcal{T}$ with $\mathcal{N}:=\flw Y{N_{\Gamma}}{}$
and $\mathcal{T}:=\flw Z{T_{\Gamma}}{}$. From Proposition~\ref{pro:chg_coord}
we deduce the expressions 
\begin{eqnarray*}
\mathcal{N}^{*}Y & = & \frac{1}{1+Y\cdot N_{\Gamma}}Y\\
\mathcal{T}^{*}Y & = & \exp\left(\delta T_{\Gamma}\right)\left(Y-\left(Y\cdot T_{\Gamma}\right)Z\right)\,.
\end{eqnarray*}
But $Y\cdot N_{\Gamma}\in\sol[][0]$ so that $\left(Y\cdot N_{\Gamma}\right)\circ\mathcal{T}=Y\cdot N_{\Gamma}$,
which yields the result.
\end{enumerate}
\end{proof}
\begin{cor}
\label{cor:eq_groupoid_Y}Outside the tangency and polar loci of $Z$
and $Y$, the sheaf of $\ww C$\textendash{}linear systems
\begin{align*}
\begin{cases}
Z\cdot T_{\Gamma} & =0\\
\delta Y\cdot T_{\Gamma} & =0\\
Y\cdot Y\cdot T_{\Gamma} & =0\\
\, & \,\\
Z\cdot N_{\Gamma}+\delta N_{\Gamma} & =0\\
Y\cdot Y\cdot N_{\Gamma} & =0
\end{cases} & \tag{\ensuremath{\star}}
\end{align*}
defines the $\mathcal{D}$\textendash{}groupoid of invariance $\aut{Z,\lie{Z,Y}}$.
Let us define the respective transverse and tangential $\mathcal{D}$\textendash{}groupoids
\begin{eqnarray*}
\trans[Z,Y] & := & \aut{Z,\lie{Z,Y}}\cap\left\{ \Gamma\,:\, T_{\Gamma}=0\right\} \\
\tang[Z,Y] & := & \aut{Z,\lie{Z,Y}}\cap\left\{ \Gamma\,:\, N_{\Gamma}=0\right\} 
\end{eqnarray*}
 so that
\begin{eqnarray*}
\aut{Z,\lie{Z,Y}} & = & \trans[Z,Y]\circ\tang[Z,Y]\,.
\end{eqnarray*}
 Then the corresponding $\mathcal{D}$\textendash{}algebroid of each
factor is a sheaf of Lie algebras of dimension at most $2$. \end{cor}
\begin{proof}
The first and third equations comes from Lemma~\ref{lem:aut_action_on_Y}~(1).
According to~(2) of said lemma the condition $\Gamma^{*}Y=d_{\Gamma}Z+c_{\Gamma}Y$,
with $\left(c_{\Gamma},d_{\Gamma}\right)\in\ww C_{\neq0}\times\ww C$
as in Lemma~\ref{lem:charac_aut}, is equivalent to 
\begin{eqnarray*}
\begin{cases}
\frac{\exp\left(\delta T_{\Gamma}\right)}{1+Y\cdot N_{\Gamma}} & =c_{\Gamma}\\
-c_{\Gamma}Y\cdot T_{\Gamma} & =d_{\Gamma}
\end{cases} &  & .
\end{eqnarray*}
Observe that because $T_{\Gamma}\in\sol[][0]$ we have $Y\cdot T_{\Gamma}\in\sol[][\delta]$
and the second equation yields $\delta Y\cdot T_{\Gamma}=0$, as well
as $Y\cdot Y\cdot T_{\Gamma}=0$. Now applying $Y\cdot$ to the first
equation we obtain $Y\cdot Y\cdot N_{\Gamma}=0$. The converse direction
is clear.

By construction $T_{\Gamma}$ and $N_{\Gamma}$ are obtained from
$\Gamma$ through the implicit function theorem, and therefore depend
meromorphically on $\Gamma$ and on none of its derivatives. Both
independent sub\textendash{}systems of~$\left(\star\right)$ characterizes
a $\mathcal{D}$\textendash{}groupoid, which has rank at most $2$
as can be checked using $Z\trs Y$.
\end{proof}

\subsubsection{Proof of Theorem~\ref{thm:compute_gal}}

~

For $\left(\delta,Y\right)\in\tx{TS}\left(Z\right)$, that is $Y\trs Z$
and $\left[Z,Y\right]=\delta Y$, define
\begin{eqnarray*}
\frak{g}_{\left(\delta,Y\right)} & := & \left\{ TZ+NY\,:\,\left(T,N\right)\mbox{ solution of }\left(\star\right)\right\} \,.
\end{eqnarray*}

\begin{lem}
\label{lem:lie_invar}The (sheaf of) Lie algebra $\frak{g}_{\left(\delta,Y\right)}$
is the $\mathcal{D}$\textendash{}algebra of $G_{\left(\delta,Y\right)}:=\aut{Z,\lie{Z,Y}}$.
In particular $X\in\frak{g}_{\left(\delta,Y\right)}$ if, and only
if,
\begin{eqnarray*}
\left[Z,X\right] & = & 0\\
\left[Y,X\right] & = & d_{X}Z+c_{X}Y\,\,\,\,\,\,\,\,\mbox{for some}\, c_{X},\, d_{X}\in\ww C
\end{eqnarray*}
with $\delta d_{X}=0$. \end{lem}
\begin{proof}
$\frak{g}_{\left(\delta,Y\right)}$ is clearly a (sheaf of) $\ww C$\textendash{}linear
space. We compute
\begin{eqnarray*}
\left[TZ+NY\,,\,\tilde{T}Z+\tilde{N}Y\right] & = & \left[TZ,\tilde{T}Z\right]+\left[TZ,\tilde{N}Y\right]-\left[\tilde{T}Z,NY\right]+\left[NY,\tilde{N}Y\right]\\
 & = & \left(\tilde{N}Y\cdot T-NY\cdot\tilde{T}\right)Z+\left(NY\cdot\tilde{N}-\tilde{N}Y\cdot N\right)Y\,.
\end{eqnarray*}
Now 
\begin{eqnarray*}
Y\cdot\left(NY\cdot\tilde{N}-\tilde{N}Y\cdot N\right) & = & 0\\
\delta Y\cdot\left(\tilde{N}Y\cdot T-NY\cdot\tilde{T}\right) & = & 0
\end{eqnarray*}
so that $\left[TZ+NY\,,\,\tilde{T}Z+\tilde{N}Y\right]\in\frak{g}_{\left(\delta,Y\right)}$
and $\frak{g}_{\left(\delta,Y\right)}$ is a Lie algebra. The fact
it is the $\mathcal{D}$\textendash{}algebra of $G_{\left(\delta,Y\right)}$
is an immediate consequence of the system~$\left(\star\right)$ already
being linear. Take now $X=TZ+NY\in\frak{g}_{\left(\delta,Y\right)}$;
a trivial computation ensures 
\begin{eqnarray*}
\left[Z,X\right] & = & \left(Z\cdot N\right)Y+N\left[Z,Y\right]=0\\
\left[Y,X\right] & = & \left(Y\cdot T\right)Z+\left(Y\cdot N-\delta T\right)Y\,.
\end{eqnarray*}
Because $Y\cdot N-\delta T\in\sol[][0]\cap\sol[Y][0]$ we deduce that
the function is a constant $c_{X}$. Likewise $Z\cdot Y\cdot T=\delta Y\cdot T=0=Y\cdot Y\cdot T$
implies $d_{X}:=Y\cdot T$ is constant. This constant vanishes if
$\delta\neq0$. The converse is clear.
\end{proof}
Define
\begin{eqnarray*}
\hat{G} & := & \bigcap_{\left(\delta,Y\right)\in\tx{TS}\left(Z\right)}G_{\left(\delta,Y\right)}
\end{eqnarray*}
 and the corresponding $\mathcal{D}$\textendash{}algebroid $\hat{\frak{g}}$
which, because of Corollary~\ref{cor:eq_groupoid_Y}, is a sheaf
of Lie algebras of dimension at most 3. Indeed the hypothesis that
the delta\textendash{}lattice $\Delta\left(Z\right)$ differs from
$\left\{ 0\right\} $ implies that 
\begin{eqnarray*}
\hat{\frak{g}}\cap\mathcal{O}Z & = & \ww CZ\,.
\end{eqnarray*}

The inclusion $\gal Z<\hat{G}$ follows from the definition. The reverse
inclusion is obtained by considering a\emph{ $\mathcal{D}$\textendash{}}groupoip
$G$ whose $\mathcal{D}$\textendash{}algebroid $\frak{g}$ contains
$Z$ and by showing that $\hat{g}\subset\frak{g}$. Let $E\neq0$
be a local linear partial differential equation of order $k$ partially
defining $\frak{g}$, that is 
\begin{eqnarray*}
E\left(T,N\right) & = & \sum_{\left|\alpha\right|\leq k}\lambda_{\alpha}\partial_{\alpha}T+\mu_{\alpha}\partial_{\alpha}N\,,
\end{eqnarray*}
where $k\in\ww Z_{\geq0}$, $\alpha\in\ww Z_{\geq0}^{2}$, $\lambda_{\alpha}$
and $\mu_{\alpha}$ belong to the sheaf of meromorphic functions near
$\left(0,0\right)$, such that 
\begin{eqnarray*}
TZ+NY\in\frak{g} & \Longrightarrow & E\left(T,N\right)=0\,.
\end{eqnarray*}
We assume that this equation is given near a point $p$ outside the
polar locus of $Z$ and $Y$, as well as outside their locus of tangency. 

Because $\ww CZ<\frak{g}$ we have
\begin{eqnarray*}
E\left(0,N\right) & = & 0
\end{eqnarray*}
whenever $NY\in\frak{g}$. We can find local analytic coordinates
$\left(t,z\right)$ in which $Z=\pp t$ and $Y=e^{\delta t}\pp z$.
Because $\left[Z,NY\right]=0$ we have
\begin{eqnarray*}
N\left(t,z\right) & = & e^{-\delta t}\eta\left(z\right)
\end{eqnarray*}
with $\eta$ meromorphic near $0$, and
\begin{eqnarray*}
\partial^{\left(n,m\right)}N\left(t,z\right) & = & \left(-\delta\right)^{n}e^{-\delta t}\eta^{\left(m\right)}\left(z\right)\,.
\end{eqnarray*}
The relation $E\left(0,N\right)=0$ corresponds therefore to some
linear differential equation $\tilde{E}\left(\eta\right)=0$ of a
$\mathcal{D}$\textendash{}algebroid on a transverse disk. According
to Casale\textendash{}Malgrange's theorem it contains the transverse
$\mathcal{D}$\textendash{}algebroid of the foliation. More precisely
the construction of G.~\noun{Casale} carried out in~\cite[Théorème 3.2]{Casa}
allows to recover a meromorphic vector field $\tilde{Y}$ near $\left(0,0\right)$
from $\tilde{E}$ such that $\left[Z,\hat{Y}\right]=\hat{\delta}\hat{Y}$
for some $\hat{\delta}\in\ww C$. If we assume $NY\in\hat{\frak{g}}$
then in particular $N\in\frak{g}_{\left(\hat{\delta},\hat{Y}\right)}$
and it also solves the equation $\tilde{E}$ in the local coordinates
$\left(t,z\right)$. We finally proved $\hat{\frak{g}}\subset\frak{g}$.

\bibliographystyle{preprints}
\bibliography{biblio}

\end{document}